\documentclass[11pt]{amsart}
\usepackage{amsfonts, amssymb, amsmath, amsthm, hyperref, color, float,enumerate}
\usepackage{url}

\theoremstyle{plain}
   \newtheorem{teo}{Theorem}
   \newtheorem{coro}[teo]{Corollary}
   \newtheorem{lema}[teo]{Lemma}
   \newtheorem{propo}[teo]{Proposition}
   
\theoremstyle{definition}
   
\theoremstyle{remark}
 \newtheorem{obs}{Remark}
 
 \newtheorem{afirmacion}{Claim}

\numberwithin{equation}{section}
\allowdisplaybreaks

\begin{document}
	
	\title[New mixed inequalities for maximal operators]{From $A_1$ to $A_\infty$: new mixed inequalities for certain maximal operators}

	\author[F. Berra]{Fabio Berra}
	\address{CONICET and Departamento de Matem\'{a}tica (FIQ-UNL),  Santa Fe, Argentina.}
	\email{fberra@santafe-conicet.gov.ar}
	
%

	\thanks{The author was supported by CONICET and UNL}
	
	
	
	\begin{abstract}
		In this article we prove mixed inequalities for maximal operators associated to Young functions, which are an improvement of a conjecture established in \cite{Berra}. Concretely, given $r\geq 1$, $u\in A_1$, $v^r\in A_\infty$ and a Young function $\Phi$ with certain properties, we have that inequality
		\[uv^r\left(\left\{x\in \mathbb{R}^n: \frac{M_\Phi(fv)(x)}{M_\Phi v(x)}>t\right\}\right)\leq C\int_{\mathbb{R}^n}\Phi\left(\frac{|f(x)|}{t}\right)u(x)v^r(x)\,dx\]
		holds for every positive $t$. The involved operator $\frac{M_\Phi(fv)(x)}{M_\Phi v(x)}$ seems to be an adequate extension when $v^r\in A_\infty$, since when we assume $v^r\in A_1$ we can replace $M_\Phi v$ by $v$, yielding a mixed inequality for $M_\Phi$ proved in \cite{Berra-Carena-Pradolini(MN)}.
		
		As an application, we furthermore exhibe and prove mixed inequalities for the generalized fractional maximal operator $M_{\gamma,\Phi}$, where $0<\gamma<n$ and $\Phi$ is a Young function of $L\log L$ type.

	\end{abstract}

	\maketitle
	
	\section*{Introduction}
	
	One of the most classical and extensively studied problem in Harmonic Analysis is the characterization of all the functions $w$ for which the Hardy-Littlewood maximal operator is bounded in $L^p(w)$, for $1<p<\infty$. This problem was first solved by B. Muckenhoupt in \cite{Muck72}, where the author proved that the inequality
	\begin{equation}\label{eq: intro - tipo fuerte de M}
	\int_{\mathbb{R}}(Mf(x))^pw(x)\,dx\leq C\int_{\mathbb{R}}|f(x)|^pw(x)\,dx
	\end{equation}
	holds for $1<p<\infty$ if and only if $w\in A_p$. Later on, this result was extended to higher dimensions and even to spaces of homogeneous type. 
	
	It is well known that, for the limit case $p=1$, the inequality above is not true. Instead, we have the estimate
	\[w(\left\{x\in \mathbb{R}^n: Mf(x)>t\right\})\leq \frac{C}{t}\int_{\mathbb{R}^n} |f(x)|w(x)\,dx.\]

	In \cite{Sawyer} Sawyer proved that if $u,v$ are $A_1$ weights, then the estimate
	\begin{equation}\label{eq: intro - desigualdad mixta Sawyer}
	uv\left(\left\{x\in \mathbb{R}: \frac{M(fv)(x)}{v(x)}>t\right\}\right)\leq \frac{C}{t}\int_{\mathbb{R}} |f(x)|u(x)v(x)\,dx
	\end{equation}
	holds for every positive $t$. From now on, we will refer to this type of inequalities as mixed because of the interaction of two different weights in it. This estimate can be seen as the weak $(1,1)$ type of the operator $Sf=M(fv)/v$, with respecto to the measure $d\mu(x)=u(x)v(x)\,dx$. One of the motivation to study this kind of estimate was the fact that inequality \eqref{eq: intro - desigualdad mixta Sawyer} combined with the Jones' factorization theorem and the Marcinkiewicz interpolation theorem allows to give, in a very easy way, a proof of \eqref{eq: intro - tipo fuerte de M} when we assume $w\in A_p$. 
	
	The proof of \eqref{eq: intro - desigualdad mixta Sawyer} is, however, a bit tricky. Since $S$ can be seen as the product of two functions, this produces a perturbation of the level sets of $M$ by the weight $v$. So it is not clear that classical covering lemmas or decomposition techniques work in this case. To overcome this difficulty, the author uses a decomposition of level sets into an adequate class of intervals with certain properties, called ``principal intervals'', an idea that was already used to prove some weak estimates previously in \cite{M-W-76}. 
	
	It was also cojectured in \cite{Sawyer} that an analogous estimate to \eqref{eq: intro - desigualdad mixta Sawyer} should still hold for the Hilbert transform.
	
	This claim was proved twenty years later by Cruz-Uribe, Martell and P\'erez in \cite{CruzUribe-Martell-Perez}. In this paper mixed weak inequalities were given, generalizing \eqref{eq: intro - desigualdad mixta Sawyer} to $\mathbb{R}^n$, not only for $M$ but also for Calder\'on-Zygmund operators (CZO) and proving the conjecture made by Sawyer. The authors considered two different type of hypotheses on the weights $u$ and $v$: $u,v\in A_1$ and $u\in A_1$ and $v\in A_\infty(u)$. For the first condition, the proof follows similar lines as in \cite{Sawyer}. On the other hand, the second condition is more suitable, since it implies that the product $uv$ belongs to $A_\infty$ and therefore is a doubling measure. This allows to apply classical techniques, like Calder\'on-Zygmund decomposition to achieve the estimate. Another conjecture arose from this work: the authors claimed that the mixed estimate
	\[uv\left(\left\{x\in \mathbb{R}^n: \frac{M(fv)(x)}{v(x)}>t\right\}\right)\leq \frac{C}{t}\int_{\mathbb{R}^n} |f(x)|u(x)v(x)\,dx\]
	should still hold under the weaker assumption $u\in A_1$ and $v\in A_\infty$. It is easy to note that both conditions on the weights above imply it. This conjecture was recently proved in \cite{L-O-P}, where the authors apply the ``principal cubes'' decomposition with adequate modifications that avoid to use the $A_1$ condition on the weight $v$.
	
	Mixed weak estimates have also been explored for a more general class of maximal functions, such as the operator $M_\Phi$ associated to the Young function $\Phi$ and defined by
	\[M_\Phi f(x)=\sup_{Q \ni x}\|f\|_{\Phi,Q},\]
	where the supremum is taken over averages of Luxemburg type (see section below for details).
	For instance, in \cite{Berra-Carena-Pradolini(M)} it was proved that if $\Phi(t)=t^r(1+\log^+t)^\delta$, with $r\geq 1$ and $\delta\geq 0$, $u\geq 0$, $v=|x|^{\beta}$ with $\beta<-n$ and $w=1/\Phi(1/v)$  then the estimate
	\[uw\left(\left\{x\in \mathbb{R}^n: \frac{M_\Phi(fv)(x)}{v(x)}>t\right\}\right)\leq C\int_{\mathbb{R}^n}\Phi\left(\frac{|f(x)|v(x)}{t}\right)Mu(x)\,dx\]
	holds for every positive $t$.  
	
	The same estimate is true if we consider this family of Young functions and two weights $u$ and $v$ such that $u,v^r\in A_1$. This result is contained in \cite{Berra}, and generalizes the corresponding estimate given in \cite{CruzUribe-Martell-Perez} for the case $\Phi(t)=t$. However, this inequality turns out to be non-homogeneous in the weight $v$. This problem was overcome in \cite{Berra-Carena-Pradolini(MN)}, where the authors proved that under the same condition on the weights and even for a more general class of Young functions $\Phi$ the inequality
	\begin{equation}\label{eq: intro - desigualdad mixta BCP (MN)}
	uv^r\left(\left\{x\in \mathbb{R}^n: \frac{M_\Phi(fv)(x)}{v(x)}>t\right\}\right)\leq C\int_{\mathbb{R}^n}\Phi\left(\frac{|f(x)|}{t}\right)u(x)v^r(x)\,dx
	\end{equation}
	holds for every positive $t$. In the same way that inequality \eqref{eq: intro - desigualdad mixta Sawyer} allows to obtain a different proof of the boundedness of $M$, the estimate above can be used to give an alternative proof of the boundedness of $M_\Phi$, when $\Phi$ is a Young function of $L\textrm{log}L$ type. 
	
	By virtue of the extension proved in \cite{L-O-P} and of the discussed results above, a natural interrogant arises: does inequality \eqref{eq: intro - desigualdad mixta BCP (MN)} hold in the more general context $u\in A_1$ and $v^r\in A_\infty$? This fact is an improvement of a conjecture established in \cite{Berra}, which remained open until now.
	
	\medskip
	
	In this paper we answer the question positively. We will be dealing with a family of Young functions with certain properties, as follows. Given $r\geq 1$ we say that a Young function belongs to the family $\mathfrak{F}_r$ if $\Phi$ is submultiplicative, has lower type $r$ and satisfies the condition
	\[\frac{\Phi(t)}{t^r}\leq C_0 (\log t)^\delta,\quad \textrm{ for }t\geq t^*\]
	for some constants $C_0>0$, $\delta\geq 0$ and $t^*\geq 1$. There is, however, an important modification in the involved operator because of the $A_\infty$ condition of $v$. Notice that when $v\in A_1$ the operator $S$ defined by Sawyer in \cite{Sawyer} is equivalent to 
	\[\mathcal{T} f(x)=\frac{M(fv)(x)}{M v(x)}.\]
	For the case of $M_\Phi$, it can be proven that $M_\Phi v\approx v$ when $v^r\in A_1$. Therefore inequality \eqref{eq: intro - desigualdad mixta BCP (MN)} can be rewritten as follows
	
	\begin{equation}\label{eq - intro - desigualdad mixta BCP(MN) alternativa}
	uv^r\left(\left\{x\in \mathbb{R}^n: \frac{M_\Phi(fv)(x)}{M_\Phi v(x)}>t\right\}\right)\leq C\int_{\mathbb{R}^n}\Phi\left(\frac{|f(x)|}{t}\right)u(x)v^r(x)\,dx.
	\end{equation}
	
	It is not difficult to see that the operator $S f$ defined by Sawyer in \cite{Sawyer} is bounded in $L^\infty(uv)$ under the assumptions on these weights. The same statement is true if we consider the operator $S_\Phi=M_\Phi(fv)/v$ that appears in \eqref{eq: intro - desigualdad mixta BCP (MN)} and the space $L^\infty(uv^r)$, but this is not clear when we have $v^r\in A_\infty$. However, if we modify the operator by considering the variant 
	\[\mathcal{T}_\Phi f(x)=\frac{M_\Phi(fv)(x)}{M_\Phi v(x)}\]
	we have an alternative operator that is bounded in $L^\infty(uv^r)$ when we only assume $v^r\in A_\infty$. This new operator seems to be a well extension for the case $v\in A_\infty$ since it is continuous in $L^\infty(uv^r)$, a property that will be useful later. Concretely, our main result is the following.
	
	\begin{teo}\label{teo: teorema principal}
		Let $r\geq 1$ and $\Phi\in \mathfrak{F}_r$. If $u\in A_1$ and $v^r\in A_\infty$ then there exists a positive constant $C$ such that the inequality
		\[uv^r\left(\left\{x\in \mathbb{R}^n: \frac{M_\Phi(fv)(x)}{M_\Phi v(x)}>t
		\right\}\right)\leq C\int_{\mathbb{R}^n}\Phi\left(\frac{|f(x)|}{t}\right)u(x)v^r(x)\,dx,\]
		holds for every positive $t$.
	\end{teo}

	In fact, we can prove the following stronger version of the inequality above:
	\begin{equation}\label{eq - intro - teorema principal, version mas fuerte}
uv^r\left(\left\{x\in \mathbb{R}^n: \frac{M_\Phi(fv)(x)}{v(x)}>t
\right\}\right)\leq C\int_{\mathbb{R}^n}\Phi\left(\frac{|f(x)|}{t}\right)u(x)v^r(x)\,dx,
	\end{equation}
	which directly implies Theorem~\ref{teo: teorema principal} since $v(x)\lesssim M_\Phi v(x)$. The key for the proof of this inequality is to combine some ideas that appear in \cite{L-O-P} with a subtle H\"{o}lder inequality that allows to split the expression $\Phi(|f|v)$ into $\Phi(|f|)v^r$. The proof also follows the ``principal cubes'' decomposition made by Sawyer in \cite{Sawyer}.
	
	Recall that we are considering the modified operator $\mathcal{T}_\Phi$, which coincides with $S_\Phi$ when $v^r$ belongs to $A_1$. The advantage of dealing with it is that we can give a version of Theorem~\ref{teo: teorema principal} where we weaken a bit the assumption on $\Phi$. That is, if $\Psi$ is an arbitrary Young function that behaves like one in the family $\mathfrak{F}_r$ but only for large $t$, we have the following result.
	
	\begin{coro}\label{coro: corolario del teorema principal}
		Let $r\geq 1$, $\Phi\in \mathfrak{F}_r$, $u\in A_1$ and $v^r\in A_\infty$. Let $\Psi$ be a Young function that verifies $\Psi(t) \approx \Phi(t)$, for every $t\geq t^*\geq 0$. Then, there exist two positive constants $C_1$ and $C_2$ such that the inequality
		\[uv^r\left(\left\{x\in \mathbb{R}^n: \frac{M_\Psi(fv)(x)}{M_\Psi v(x)}>t
		\right\}\right)\leq C_1\int_{\mathbb{R}^n}\Psi\left(\frac{C_2|f|}{t}\right)uv^r\]
		holds for every $t>0$. 
	\end{coro}

Corollary~\ref{coro: corolario del teorema principal} will play a fundamental role to obtain mixed inequalities for a fractional version of the operators considered above. That is, as an important application of these results we can give mixed estimates for the \emph{generalized fractional maximal operator} $M_{\gamma,\Phi}$,   defined by
\[M_{\gamma,\Phi}f(x)=\sup_{Q\ni x}|Q|^{\gamma/n}\|f\|_{\Phi,Q},\]
where $0<\gamma<n$ and $\Phi$ is a Young function.

When we consider the family of functions $\Phi(t)=t^r(1+\log^+t)^\delta$, for $r\geq 1$ and $\delta\geq 0$, the operator $M_{\gamma,\Phi}$ is bounded from $L^p(w^p)$ to $L^q(w^q)$ if and only if $w^r\in A_{p/r,q/r}$, for $0<\gamma<n/r$, $r<p<n/\gamma$ and $1/q=1/p-\gamma/n$. This result was set and proved in \cite{B-D-P}, and generalizes the strong $(p,q)$ type of $M_\gamma$ between Lebesgue spaces when we consider $r=1$ and $\delta=0$. As it occurs with $M_\gamma$, this estimate fails in the limit case $p=r$. In \cite{G-P-S-Sinica} it was proved an endpoint weak type estimate for this operator in the setting of spaces of homogeneous type. The corresponding inequality for the euclidean case is
\begin{equation}\label{eq - intro - tipo debil modular de M_{gamma, Phi}}
w\left(\left\{x\in \mathbb{R}^n: M_{\gamma,\Phi}f(x)>t\right\}\right)\leq C\varphi\left(\int_{\mathbb{R}^n}\Phi\left(\frac{|f(x)|}{t}\right)\Psi(M w(x))\,dx\right),
\end{equation}
where $\Psi(t)=t^{n-r\gamma}(1+\log^+t^{-r\gamma/n})^{\delta}$ and $\varphi(t)=(t(1+\log^+t^{r\gamma/n})^\delta)^{n/(n-r\gamma)}$.

In \cite{FB} we study mixed inequalities for this operator when a power of the weight $v$ belongs to $A_1$. These inequalities arose from the fact that those mixed inequalities allow to obtain an alternative proof of the continuity properties of $M_{\gamma, \Phi}$ discussed above (see Section 4.3 in \cite{FB} for further details). The proof relies on a pointwise estimate that relates the operators $M_{\gamma,\Phi}$ and $M_\Phi$ (see section below) and it is a generalization of a Hedberg type inequality used in \cite{Berra-Carena-Pradolini(J)} to give mixed estimates for the fractional maximal operator $M_\gamma$, when $0<\gamma<n$.

The following two theorems contain mixed inequalities for $M_{\gamma,\Phi}$ for the cases $r<p<n/\gamma$ and the limit case $p=r$, respectively. In the particular case in which the power of $v$ belongs to $A_1$, both results were established and proved in \cite{FB}. 

\begin{teo}	\label{teo: mixta para M_{gamma,Phi}, caso r<p<n/gamma}
	Let $\Phi(t)=t^r(1+\log^+ t)^\delta$, with $r\geq 1$ and $\delta \geq 0$. Let $0<\gamma<n/r$, $r<p<n/\gamma$ and $1/q=1/p-\gamma/n$. If $u\in A_1$ and $v^{q(1/p+1/r')}\in A_\infty$, then we have that
	\[uv^{q(1/p+1/r')}\left(\left\{x\in \mathbb{R}^n: \frac{M_{\gamma,\Phi}(fv)(x)}{M_\eta v(x)}>t\right\}\right)^{1/q}\leq C\left[ \int_{\mathbb{R}^n}\left(\frac{|f(x)|}{t}\right)^pu^{p/q}(x)(v(x))^{1+p/r'}\,dx\right]^{1/p},\]
	where $\eta(t)=t^{q/p+q/r'}(1+\log^+ t)^{n\delta/(n-r\gamma)}$.
\end{teo}

\begin{teo}\label{teo: mixta para M_{gamma,Phi}, caso p=r}
	Let $\Phi(t)=t^r(1+\log^+t)^\delta$, with $r\geq 1$ and $\delta\geq 0$. Let $0<\gamma<n/r$ and $1/q=1/r-\gamma/n$. If $u\in A_1$ and $v^q\in A_\infty$, then there exists a positive constant $C$ such that
	\[uv^q\left(\left\{x\in \mathbb{R}^n: \frac{M_{\gamma,\Phi}(fv)(x)}{M_\eta v(x)}>t\right\}\right)\leq \varphi\left(\int_{\mathbb{R}^n}\Phi_\gamma\left(\frac{|f(x)|}{t}\right)\Psi\left(u^{1/q}(x)v(x)\right)\,dx\right),\]
	where $\eta(t)=t^q(1+\log^+t)^{n\delta/(n-r\gamma)}$, $\varphi(t)=[t(1+\log^+t)^\delta ]^{q/r}$, $\Psi(t)=t^r(1+\log^+(t^{1-q/r}))^{n\delta/(n-r\gamma)}$ and $\Phi_\gamma(t)=\Phi(t)(1+\log^+t)^{\delta r\gamma/(n-r\gamma)}$. 
\end{teo}

\begin{obs}
These two last theorems are very important since they give well extensions to many results concerning to mixed inequalities.
\begin{itemize}
	\item In \cite{Berra-Carena-Pradolini(J)} we use mixed inequalities for $M$ proved in \cite{CruzUribe-Martell-Perez} to give the corresponding estimates for the fractional maximal operator $M_\gamma$. A stronger version can be obtained if we use in the proof the mixed estimate for $M$ given in \cite{L-O-P}. However, this can only be done for the limit case $p=1$ and $q=n/(n-\gamma)$, since the proof for the remaining cases $1<p<n/\gamma$ depends heavily on an auxiliary lemma which uses $A_1$ condition of $v^{q/p}$. We can use Theorem~\ref{teo: mixta para M_{gamma,Phi}, caso r<p<n/gamma} to overcome this problem. Indeed, if we set $r=1$ and $\delta=0$ we have $\Phi(t)=t$. Therefore $M_{\gamma,\Phi}=M_\gamma$, and we obtain the corresponding extension of Theorem 1 in \cite{Berra-Carena-Pradolini(J)} for the case $v^{q/p}\in A_\infty$, for every $1<p<n/\gamma$. 
	
	\item If we assume $v^{q(1/p+1/r')}\in A_1$ then $M_\eta v \approx v$. In this case we precisely obtain Teorema~4.9 in \cite{FB}. If we furthermore set $r=1$ and $\delta=0$, then $M_\eta v = M_{q/p} v \approx v$. This recovers Theorem 1 in \cite{Berra-Carena-Pradolini(J)}, for the case $v^{q/p}\in A_1$. 
	\item By virtue of \eqref{eq - intro - teorema principal, version mas fuerte}, Theorem~\ref{teo: mixta para M_{gamma,Phi}, caso p=r} is still true when we replace $M_\eta v$ by $v$. This extends the limit case $p=1$ and $q=n/(n-\gamma)$ for $M_\gamma$ in \cite{Berra-Carena-Pradolini(J)} when we set $\Phi(t)=t$.
	\item If we assume $v^q\in A_1$ in Theorem~\ref{teo: mixta para M_{gamma,Phi}, caso p=r}, we get $M_\eta v\approx v$ and this recovers Teorema~4.11 in \cite{FB}.
	\item When we take $v=1$ in Theorem~\ref{teo: mixta para M_{gamma,Phi}, caso p=r} we obtain an estimate similar to \eqref{eq - intro - tipo debil modular de M_{gamma, Phi}}.
\end{itemize}
\end{obs}

The remainder of this paper is organized as follows: in \S~\ref{section: preliminares} we give the required preliminaries and basic definitions. \S~\ref{section: resultados auxiliares} contains some auxiliary results that will be useful in the main proofs. In \S~\ref{section: prueba principal} we prove both Theorem~\ref{teo: teorema principal} and Corollary~\ref{coro: corolario del teorema principal}. Finally, we prove Theorem~\ref{teo: mixta para M_{gamma,Phi}, caso r<p<n/gamma} and Theorem~\ref{teo: mixta para M_{gamma,Phi}, caso p=r} in \S~\ref{section: aplicaciones}, as an application of the main result.

\section{Preliminaries and definitions}\label{section: preliminares}

We shall say that $A\lesssim B$ if there exists a positive constant $C$ such that $A\leq C B$. The constant $C$ may change on each occurrence. If we have $A\lesssim B$ and $B\lesssim A$, this will be denoted as $A\approx B$.

Given a function $\varphi$, we will say that $f\in L^{\varphi}_{loc}$ if $\varphi(|f|)$ is locally integrable. In the case $\varphi(t)=t$, the corresponding space is the usual $L^1_{loc}$.

By a \emph{weight} $w$ we understand a function that is locally integrable, positive and finite in almost every $x$. Given $1<p<\infty$, the $A_p$-Muckenhoupt class is defined to be the set of weights $w$ that verify 
\[\left(\frac{1}{|Q|}\int_Qw\right)\left(\frac{1}{|Q|}\int_Qw^{1-p'}\right)^{p-1}\leq C,\]
for some positive constant $C$ and for every cube $Q\subseteq \mathbb{R}^n$. We shall consider cubes in $\mathbb{R}^n$ with sides parallel to the coordinate axes.
In the limit case $p=1$, we say that $w\in A_1$ if there exists a positive constant $C$ such that for every cube $Q$
\[\frac{1}{|Q|}\int_Qw \leq C\inf_Qw,\]
where $\inf_Q$ denotes the essential infimum of $w$ in $Q$.

The smallest constants $C$ for which the corresponding inequalities above hold are denoted by $[w]_{A_p}$, $1\leq p< \infty$ and called the characteristic $A_p$ constants of $w$.

Finally, the $A_\infty$ class is defined as the collection of all the $A_p$ classes, that is, $A_\infty=\bigcup_{p\geq 1}A_p$. It is well known that the $A_p$ classes are increasing on $p$, that is, if $p\leq q$ then $A_p\subseteq A_q$.  For further details and other properties of weights see \cite{javi} or \cite{grafakos}.

There are many conditions that characterize $A_\infty$. In this paper we will use the following one: $w\in A_\infty$ if there exist positive constants $C$ and $\varepsilon$ such that, for every cube $Q\subseteq \mathbb{R}^n$ and every measurable set $E\subseteq Q$ we have
\[\frac{w(E)}{w(Q)}\leq C\left(\frac{|E|}{|Q|}\right)^{\varepsilon},\]
where $w(E)=\int_E w$.

Every Muckenhoupt weight satisfies a  \emph{reverse H\"{o}lder condition}. That is, if
$w\in A_p$ for some $1\leq p<\infty$, then there exist positive constants $C$
and $s>1$ that depend only on the dimension $n$, $p$ and
$[w]_{A_p}$, such that 
\begin{equation*}
\left(\frac{1}{|Q|}\int_Q w^s(x)\,dx\right)^{1/s}\leq
\frac{C}{|Q|}\int_Q w(x)\,dx
\end{equation*}
for every cube $Q$. We write $w\in
\textrm{RH}_s$ to indicate that the inequality above holds,
and we denote by $[w]_{\textrm{RH}_s}$ the smallest constant $C$ associated to this condition.
It is easy to see that
$\textrm{RH}_s\subseteq \textrm{RH}_q$,
for every $1<q<s$.\\

Given a locally integrable function $f$, the \emph{Hardy-Littlewood maximal operator} is defined by
\[Mf(x)=\sup_{Q\ni x}\frac{1}{|Q|}\int_Q |f(y)|\,dy.\]

We say that $\varphi:[0,\infty)\to[0,\infty]$ is a \emph{Young function} if it is convex, increasing, $\varphi(0)=0$ and $\varphi(t)\to\infty$ when $t\to\infty$. Given a Young function $\varphi$, the maximal operator $M_\varphi$ is defined, for $f\in L^\varphi_{\textit{loc}}$, by
\[M_\varphi f(x)=\sup_{Q\ni x}\left\|f\right\|_{\varphi,Q},\]
where $\left\|f\right\|_{\varphi,Q}$ denotes the \emph{Luxemburg type average} of the function $f$ in the cube $Q$, defined by
\[\left\|f\right\|_{\varphi,Q}=\inf\left\{\lambda>0 : \frac{1}{|Q|}\int_Q\varphi\left(\frac{|f(y)|}{\lambda}\right)\,dy\leq 1 \right\}.\]

Given a weight $w$, we can also consider the \emph{weighted Luxemburg type average}
$\left\|f\right\|_{\varphi,Q,w}$ to be defined as
\[\left\|f\right\|_{\varphi,Q,w}=\inf\left\{\lambda>0 : \frac{1}{w(Q)}\int_Q\varphi\left(\frac{|f(y)|}{\lambda}\right)w(y)\,dy\leq 1 \right\}.\]

It is easy to check that from the definition above we have
\[\frac{1}{w(Q)}\int_Q\varphi\left(\frac{|f(y)|}{\left\|f\right\|_{\varphi,Q,w}}\right)w(y)\,dy\leq 1.\]

When $w\in A_\infty$, the measure given by $d\mu(x)=w(x)\,dx$ is doubling. Thus, by following the same arguments as in the result of Krasnosel'ski{\u\i} and Ruticki{\u\i} (\cite{KR}, see also \cite{raoren}) we can get that
\begin{equation}\label{eq: preliminares - equivalencia normal Luxemburo con infimo}
\|f\|_{\varphi,Q,w}\approx\inf_{\tau>0}\left\{\tau+\frac{\tau}{w(Q)}\int_{Q}\varphi\left(\frac{|f(x)|}{\tau}\right)w(x)\,dx\right\}.
\end{equation}

A Young function $\varphi$ is \textit{submultiplicative} if there exists a positive constant $C$ such that
\[\varphi(st)\leq C\varphi(s)\varphi(t)\]
for every $s,t\geq 0$. 
We say $\varphi$ has \textit{lower type} $p$, $0<p<\infty$ if there exists a positive constant $C_p$ such that
\[\varphi(st)\leq C_ps^p\varphi(t),\]
for every $0<s\leq 1$ and $t>0$. Also, $\varphi$ has \textit{upper type} $q$, $0<q<\infty$ if there exists a positive constant $C_q$ such that
\[\varphi(st)\leq C_qs^q\varphi(t),\]
for every $s\geq 1$ and $t>0$. As an immediate consequence of these definitions we have that, if $\varphi$ has lower type $p$ then $\varphi$ has lower type $\tilde p$, for every $0<\tilde p<p$. Also, if $\varphi$ has upper type $q$, then it has upper type $\tilde q$, for every $\tilde q>q$.

Given a function $\varphi:[0,\infty)\to[0,\infty]$ we define the \emph{generalized inverse} of $\varphi$ as
\[\varphi^{-1}(t)=\inf\{s\geq 0: \Phi(s)\geq t\},\]
with the convention that $\inf \emptyset =\infty$.

The \emph{generalized H\"{o}lder inequality} establishes that if $\varphi, \psi$ and $\phi$ are Young functions satisfying
\[\psi^{-1}(t)\phi^{-1}(t)\lesssim \varphi^{-1}(t)\]
for every $t\geq t^*>0$ then there exists a positive constant $C$ such that
\begin{equation}\label{eq: desigualdad de Hölder generalizada}
\|fg\|_{\varphi,Q}\leq C\|f\|_{\psi,Q}\|g\|_{\phi,Q}.
\end{equation}

\medskip

In this article we shall deal with Young functions of the type $\varphi(t)=t^r(1+\log^+t)^\delta$, where $r\geq 1$, $\delta \geq 0$ and $\log^+t=\max\{0,\log t\}$. It is well known that this class of functions are submultiplicative, have a lower type $r$ and have upper type $q$, for every $q>r$. Moreover, we have (see, for example, Proposici\'on~1.18 in \cite{FB}) that
\begin{equation}\label{eq: inversa generalizada de Phi}
\varphi^{-1}(t)\approx t^{1/r}(1+\log^+t)^{-\delta/r}.
\end{equation}


\medskip

The proof of the main result can be reduced to study the dyadic version of the operator involved. 
 By a \emph{dyadic grid} $\mathcal{D}$ we understand a collection of cubes of $\mathbb{R}^n$ that satisfies the following properties:
\begin{enumerate}
	\item every cube $Q$ in $\mathcal{D}$ has side length $2^k$, for some $k\in\mathbb{Z}$;
	\item if $P\cap Q\neq\emptyset$ then $P\subseteq Q$ or $Q\subseteq P$;
	\item $\mathcal{D}_k=\{Q\in\mathcal{D}: \ell(Q)=2^k\}$ is a partition of $\mathbb{R}^n$ for every $k\in\mathbb{Z}$, where $\ell(Q)$ denotes the side length of $Q$.
\end{enumerate}

The dyadic maximal operator $M_{\varphi,\mathcal{D}}$ associated to the Young function $\varphi$ and to the dyadic grid $\mathcal{D}$ is defined in a similar way as above, but the supremum is taken over all cubes in $\mathcal{D}$.
It can be shown that
\begin{equation}\label{eq - intro - control diadico de Mphi}
M_\Phi f(x)\leq C\sum_{i=1}^{3^n}M_{\Phi, \mathcal{D}^{(i)}}f(x),
\end{equation}
where $\mathcal{D}^{(i)}$ are fixed dyadic grids.

%
%


\section{Auxiliary results}\label{section: resultados auxiliares}

The following lemma gives us the decomposition of level sets of dyadic generalized maximal operators into dyadic cubes. A proof of this result can be found in (\cite{Berra}, Lemma 2.1).

\begin{lema}\label{lema: descomposicion de CZ del espacio}
	Given $\lambda>0$, a bounded function $f$ with compact support, a dyadic grid $\mathcal{D}$ and a Young function $\varphi$, there exists a family of maximal cubes $\{Q_j\}$ of $\mathcal{D}$ that satisfies
	\[\{x\in\mathbb{R}^n: M_{\varphi,\mathcal{D}}f(x)>\lambda\}=\bigcup_j Q_j,\]
	and $\left\|f\right\|_{\varphi,Q_j}>\lambda$  for every $j$.
\end{lema}

The next lemma is purely technical and gives a fundamental fact that will be crucial in the main proof. 

\begin{lema}\label{lema: estimacion constante epsilon}
	Let $f$ be the function defined in $[0,\infty)$ by 
	\[f(x)=\left\{\begin{array}{ccl}
	\left(1+\frac{1}{x}\right)^{\frac{x}{1+x}}& \textrm{ if } & x>0,\\
	1 & \textrm{ if} & x=0.
	
	\end{array}
	\right.\]
	Then we have that $1\leq f(x)\leq e^{1/e}$, for every $x\geq 0$.
\end{lema}

The following proposition establishes that if $\Psi$ and $\Phi$ are equivalent Young functions for $t$ large, then they have equivalent Luxemburg norm on every cube $Q$. As a consequence, we have that $M_\Psi\approx M_\Phi$.

\begin{propo}\label{propo: normas de funciones equivalentes para t grande son equivalentes}
	Let $\Phi$ and $\Psi$ be Young functions that verify $\Phi(t) \approx \Psi(t)$ for every $t\geq t_0\geq 0$. Then $\|{\cdot}\|_{\Phi,Q}\approx \|{\cdot}\|_{\Psi,Q}$, for every cube $Q$.
\end{propo}  

\begin{proof}
	Fix $f\in L^{\Phi}_\textit{loc}$. By hypothesis there exist two positive constants $C_1$ and $C_2$ such that
	\[C_1\Psi(t)\leq \Phi(t)\leq C_2\Psi(t),\]
	for $t\geq t_0$.
	Thus, given $\lambda>0$ we have that
	\begin{align*}
	\frac{1}{|Q|}\int_Q\Phi\left(\frac{|f|}{\lambda}\right)&=\frac{1}{|Q|}\int_{Q\cap\{|f|\leq t_0\lambda\}}\Phi\left(\frac{|f|}{\lambda}\right)+\frac{1}{|Q|}\int_{Q\cap\{|f|>t_0\lambda\}}\Phi\left(\frac{|f|}{\lambda}\right)\\
	&\leq \Phi(t_0)+\frac{C_2}{|Q|}\int_Q\Psi\left(\frac{|f|}{\lambda}\right).
	\end{align*}
	If we set $\lambda=\|{f}\|_{\Psi,Q}$ then
	\[\frac{1}{|Q|}\int_Q\Phi\left(\frac{|f|}{\lambda}\right)\leq \Phi(t_0)+C_2,\]
	which implies that $\|{f}\|_{\Phi,Q}\leq \max\{1,\Phi(t_0)+C_2\}\|{f}\|_{\Psi,Q}$. By interchanging the roles of $\Phi$ and $\Psi$ we can obtain the other inequality. 
\end{proof}

The next result gives a version of Jensen inequality for Luxemburg averages.

\begin{lema}\label{lema: Jensen para promedios de tipo Luxemburgo}
		Let $\Phi$ be a Young function, $f\in L^{\Phi}_{\textit{loc}}$ and $r\geq 1$. Then there exists a positive constant $C$ such that for every cube $Q$
	\[\|f\|_{\Phi,Q}^r\leq C\|f^r\|_{\Phi,Q}.\]
\end{lema}

\begin{proof}
	Notice that if $t\geq 1$, then $\Phi(t^{1/r})\leq \Phi(t)$. Picking $\lambda=\|f^r\|_{\Phi,Q}^{1/r}$ we can estimate
	\begin{align*}
	\frac{1}{|Q|}\int_Q \Phi\left(\frac{|f(y)|}{\lambda}\right)\,dy&=\frac{1}{|Q|}\int_{Q\cap\{|f|\leq \lambda\}} \Phi\left(\frac{|f(y)|}{\lambda}\right)\,dy+\frac{1}{|Q|}\int_{Q\cap\{|f|> \lambda\}} \Phi\left(\frac{|f(y)|}{\lambda}\right)\,dy\\
	&\leq \Phi(1)+\frac{1}{|Q|}\int_{Q\cap\{|f|> \lambda\}} \Phi\left(\left(\frac{|f(y)|^r}{\|f^r\|_{\Phi,Q}}\right)^{1/r}\right)\,dy\\
	&\leq \Phi(1)+\frac{1}{|Q|}\int_{Q} \Phi\left(\frac{|f(y)|^r}{\|f^r\|_{\Phi,Q}}\right)\,dy\\
	&\leq \Phi(1)+1,
	\end{align*}
	and therefore 
	\[\|f\|_{\Phi,Q}\leq C\|f^r\|_{\Phi,Q}^{1/r}.\qedhere\]
\end{proof}

\medskip

The following proposition provides a pointwise estimate between the operators $M_{\gamma,\Phi}$ and $M_\xi$, where the functions involved are related in certain way. This result can be seen as a Hedberg type estimate for generalized maximal operators.

\begin{propo}\label{propo: estimacion puntual M_{gamma, Phi}}
	Let $0<\gamma<n$, $1\leq p<n/\gamma$ and $1/q=1/p-\gamma/n$. Let $\Phi,\xi$ be Young functions verifying $t^{\gamma/n}\xi^{-1}(t)\leq C \Phi^{-1}(t)$, for every $t\geq t_0\geq 0$. Then, for every nonnegative functions $w$ and $f\in L^p$ we have that
	\[M_{\gamma,\Phi}\left(\frac{f}{w}\right)(x)\leq CM_\xi\left(\frac{f^{p/q}}{w}\right)(x)\left(\int_{\mathbb{R}^n}f^p(y)\,dy\right)^{\gamma/n},\]
	for every $x\in \mathbb{R}^n$. 
\end{propo}

\begin{proof}
	Define $s=1+q/p'$ and let $g=f^{p/s}w^{-q/s}$. Then
	\[\frac{f}{w}=g^{s/p}w^{q/p-1}.\]
	Fix $x$ and a cube $Q$ such that $x\in Q$. By using generalized H\"{o}lder inequality \eqref{eq: desigualdad de Hölder generalizada} we obtain
	\begin{align*}
	|Q|^{\gamma/n}\left\|\frac{f}{w}\right\|_{\Phi,Q}&=|Q|^{\gamma/n}\left\|g^{s/p}w^{q/p-1}\right\|_{\Phi,Q}\\
	&=|Q|^{\gamma/n}\left\|g^{1-\gamma/n}g^{s/p+\gamma/n-1}w^{q\gamma/n}\right\|_{\Phi,Q}\\
	&\leq C|Q|^{\gamma/n}\left\|g^{1-\gamma/n}\right\|_{\xi,Q}\left\|g^{s/p+\gamma/n-1}w^{q\gamma/n}\right\|_{L^{n/\gamma},Q}\\
	&= C\left\|f^{p/q}w^{-1}\right\|_{\xi,Q}\left(\int_Q f^p(y)\,dy\right)^{\gamma/n}\\
	&\leq C M_\xi\left(\frac{f^{p/q}}{w}\right)(x)\left(\int_{\mathbb{R}^n} f^p(y)\,dy\right)^{\gamma/n}.\qedhere
	\end{align*}
\end{proof}

%

\section{Proof of the main result}\label{section: prueba principal}

We devote this section to the proof of Theorem~\ref{teo: teorema principal} and its corollary. As we stated before, it will be enough to obtain \eqref{eq - intro - teorema principal, version mas fuerte}. We shall present and prove some auxiliary results that will be useful to this purpose.  Recall we are dealing with a function $\Phi\in \mathfrak{F}_r$, where $r\geq 1$ is given.
By \eqref{eq - intro - control diadico de Mphi}, it will be enough to prove that
	\[uv^r\left(\left\{x\in \mathbb{R}^n: \frac{M_{\Phi,\mathcal{D}}(fv)(x)}{v(x)}>t
\right\}\right)\leq C\int_{\mathbb{R}^n}\Phi\left(\frac{|f(x)|}{t}\right)u(x)v^r(x)\,dx,\]
where $\mathcal{D}$ is a given dyadic grid.
We can also assume that $t=1$ and that $g=|f|v$ is a bounded function with compact support. Then, we can write
\[uv^r\left(\left\{x\in \mathbb{R}^n: \frac{M_{\Phi,\mathcal{D}}(fv)(x)}{v(x)}>1
\right\}\right)=\sum_{k\in \mathbb{Z}} uv^r\left(\left\{x: \frac{M_{\Phi,\mathcal{D}}g(x)}{v(x)}>1, a^k<v\leq a^{k+1}
\right\}\right)=\sum_{k\in \mathbb{Z}}uv^r(E_k).\]  
For every $k\in \mathbb{Z}$ we consider the set
\[\Omega_k=\{x\in \mathbb{R}^n: M_{\Phi,\mathcal{D}}g(x)>a^k\},\]
and by virtue of Lemma~\ref{lema: descomposicion de CZ del espacio}, there exists a collection of dyadic cubes $\{Q_j^k\}_j$ that satisfies
\[\Omega_k=\bigcup_j Q_j^k,\]
and $\|g\|_{\Phi,Q_j^k}>a^k$ for each $j$. By maximality, we have
\begin{equation}\label{eq: promedios Luxemburgo de g son como a^k}
a^k<\|g\|_{\Phi,Q_j^k}\leq 2^n a^k, \quad \textrm{ for every }j.
\end{equation}
We proceed now to split for every $k\in \mathbb{Z}$, as in \cite{L-O-P}, the obtained cubes in different classes. If $\ell \in \mathbb{N}_0$, we set
\[\Lambda_{\ell,k}=\left\{Q_j^k: a^{(k+\ell)r}\leq \frac{1}{|Q_j^k|}\int_{Q_j^k} v^r< a^{(k+\ell+1)r}\right\},\]
and also
\[\Lambda_{-1,k}=\left\{Q_j^k:  \frac{1}{|Q_j^k|}\int_{Q_j^k} v^r< a^{kr}\right\}.\]

The next step is to split every cube in the family $\Lambda_{-1,k}$. Fixed $Q_j^k\in \Lambda_{-1,k}$, we perform the  Calder\'on-Zygmund decomposition of the function $v^r\mathcal{X}_{Q_j^k}$ at height $a^{kr}$. Then we obtain, for each $k$, a collection of maximal cubes, $\left\{Q_{j,i}^k\right\}_i$, contained in $Q_j^k$ and which satisfy
\begin{equation}\label{eq: promedios de v^r sobre Q_{j,i}^k son como a^{kr}}
a^{kr}<\frac{1}{|Q_{j,i}^k|}\int_{Q_{j,i}^k}v^r\leq 2^na^{kr},\quad \textrm{ for every }i.
\end{equation}

We now define, for every  $\ell\geq 0$ the sets
\[\Gamma_{\ell,k}=\left\{Q_j^k\in \Lambda_{\ell,k}: \left|Q_j^k\cap \left\{x: a^k<v\leq a^{k+1}\right\}\right|>0\right\},\]
and also
\[\Gamma_{-1,k}=\left\{Q_{j,i}^k\in \Lambda_{-1,k}: \left|Q_{j,i}^k\cap \left\{x: a^k<v\leq a^{k+1}\right\}\right|>0\right\}.\]

Since $E_k\subseteq \Omega_k$, we can estimate

\begin{align*}
\sum_{k\in \mathbb{Z}} uv^r(E_k)&=\sum_{k\in \mathbb{Z}} uv^r(E_k\cap \Omega_k)\\
&=\sum_{k\in \mathbb{Z}} \sum_j uv^r(E_k\cap Q_j^k)\\
&\leq \sum_{k\in \mathbb{Z}}\sum_{\ell\geq 0} \sum_{Q_j^k\in \Gamma_{\ell,k}}a^{(k+1)r}u(E_k\cap Q_j^k)+\sum_{k\in \mathbb{Z}}\,\, \sum_{i:Q_{j,i}^k\in \Gamma_{-1,k}}a^{(k+1)r}u(Q_{j,i}^k)\\
&=A+B.
\end{align*}

If we can prove that given a negative integer $N$, there exists a positive constant $C$, independent of $N$, for which the following estimate
\begin{equation}\label{eq: desigualdad con C independiente de N}
\sum_{k\geq N}\sum_{\ell\geq 0} \sum_{Q_j^k\in \Gamma_{\ell,k}}a^{(k+1)r}u(E_k\cap Q_j^k)+\sum_{k\geq N} \sum_{i:Q_{j,i}^k\in \Gamma_{-1,k}}a^{(k+1)r}u(Q_{j,i}^k)\leq C\int_{\mathbb{R}^n}\Phi\left(|f|\right)uv^r
\end{equation}
holds, then the proof would be completed by letting $N\to-\infty$.

\medskip

In order to prove \eqref{eq: desigualdad con C independiente de N} we need some auxiliary results. The two following lemmas deal with the family of cubes defined above. Both were set and proved in \cite{L-O-P}. However, we include the proof of the second one since there are slight changes because we work with Luxemburg averages.

\begin{lema}[\cite{L-O-P}, Lemma 2.3]\label{lema: acotacion de u(Ek cap Q_j^k)}
	If $\ell\geq 0$ and $Q_j^k\in \Gamma_{\ell,k}$, then there exist two positive constants $c_1$ and $c_2$, depending only on $u$ and $v^r$ such that
	\begin{equation}\label{eq: tesis lema: acotacion de u(Ek cap Q_j^k)}
	u(E_k\cap Q_j^k)\leq c_1e^{-c_2\ell r}u(Q_j^k).
	\end{equation}
\end{lema}

\begin{lema}\label{lema: sparsity de los cubos}
	If $Q$ is a cube in $\Gamma=\cup_{\ell\geq -1}\cup_{k\geq N} \Gamma_{\ell,k}$, then there exists a positive constant $C$, independent of $Q$, such that
	\[\left|\bigcup_{Q'\in \Gamma, Q'\subsetneq Q} Q'\right|\leq C|Q|.\]
\end{lema}

\begin{proof}
	We shall first prove that if $Q_j^k\subsetneq Q_s^t$ or $Q_j^k\subsetneq Q_{s,m}^t$ or $Q_{j,i}^k\subsetneq Q_s^t$ or $Q_{j,i}^k\subsetneq Q_{s,m}^t$, then $k>t$. By maximality, for the first case we have
	\[a^t<\|g\|_{\Phi,Q_s^t}\leq a^k,\]
	from where we easily deduce that $k>t$. The second case can be reduced to the first, since $Q_{s,m}^t\subsetneq Q_s^t$.  
	For the third case, 
	notice that $Q_j^k\neq Q_s^t$. Therefore, we must have $Q_s^t\subsetneq Q_j^k$ or $Q_j^k\subsetneq Q_s^t$. If the first condition held, we would have $t>k$. Then
	\[\frac{1}{|Q_s^t|}\int_{Q_s^t}v^r\mathcal{X}_{Q_j^k}=\frac{1}{|Q_s^t|}\int_{Q_s^t}v^r\geq a^{tr}>a^{kr},\]
	and $Q_{j,i}^k\subsetneq Q_s^t$. This is absurd because $Q_{j,i}^k$ is a maximal cube that verifies 
	\[\frac{1}{|Q_{j,i}^k|}\int_{Q_{j,i}^k}v^r>a^{kr}.\]  
	Then we must have $Q_j^k\subsetneq Q_s^t$ and this implies that $k>t$. Finally, the forth case follows from the third since $Q_{s,m}^t\subsetneq Q_s^t$. 
	
	With this fact in mind consider a cube in $\Gamma$, say, $Q_s^t$. We want to estimate
	\[\left|\bigcup_{Q\in \Gamma, Q\subsetneq Q_s^t} Q\right|.\]
	Note that if $Q\subsetneq Q_s^t$, then the level of $Q$ is greater than $t$. Therefore
	\[\left|\bigcup_{Q\in \Gamma, Q\subsetneq Q_s^t} Q\right|\leq\sum_{k>t}\sum_j |Q_j^k|.\]
	Since $a^k<\|g\|_{\Phi,Q_j^k}$ we have that
	\[1<\frac{1}{|Q_j^k|}\int_{Q_j^k}\Phi\left(\frac{g}{a^k}\right), \textrm{ or equivalently }|Q_j^k|<\int_{Q_j^k}\Phi\left(\frac{g}{a^k}\right).\]
	On the other hand, since $\|g\|_{\Phi,Q_s^t}\leq 2^na^t$ we have
	\[\frac{1}{|Q_s^t|}\int_{Q_s^t}\Phi\left(\frac{g}{2^na^t}\right)\leq 1, \textrm{ or equivalently }\int_{Q_s^t}\Phi\left(\frac{g}{2^na^t}\right)\leq |Q_s^t|.\]
	By combining these inequalities together with the convexity of $\Phi$ we obtain
	\begin{align*}
	\sum_{k>t}\sum_j |Q_j^k|&<\sum_{k>t}\sum_j\int_{Q_j^k}\Phi\left(\frac{g}{a^k}\right)\\
	&\leq\sum_{k>t}\sum_j2^na^{t-k}\int_{Q_j^k}\Phi\left(\frac{g}{2^na^t}\right)\\
	&\leq 2^n \sum_{k>t}a^{t-k}\int_{Q_s^t}\Phi\left(\frac{g}{2^na^t}\right)\\
	&\leq 2^n|Q_s^t|\sum_{k>t}a^{t-k}\\
	&=\frac{2^n}{a-1}|Q_s^t|.\qedhere
	\end{align*}
\end{proof}

\medskip

\begin{proof}[Proof of Theorem~\ref{teo: teorema principal}]
	We shall write some parts of the proof as claims, which will be proved at the end, for the sake of clearness. Recall that we have to estimate the two quantities
	\[A_N:= \sum_{k\geq N}\sum_{\ell\geq 0} \sum_{Q_j^k\in \Gamma_{\ell,k}}a^{(k+1)r}u(E_k\cap Q_j^k)\]
	and
	\[B_N:=\sum_{k\geq N} \sum_{i:Q_{j,i}^k\in \Gamma_{-1,k}}a^{(k+1)r}u(Q_{j,i}^k)\]
	by $C\int_{\mathbb{R}^n}\Phi\left(|f|\right)uv^r$, with $C$ independent of $N$.
	
	We shall start with the estimate of $A_N$. Fix $\ell\geq 0$ and let $\Delta_\ell=\cup_{k\geq N} \Gamma_{\ell,k}$. We define a sequence of sets recursively as follows: 
	\[P_0^\ell=\{Q: Q \textrm{ is maximal in } \Delta_\ell \textrm{ in the sense of inclusion}\}\]
	and for $m\geq 0$ given we say that $Q_j^k\in P_{m+1}^\ell$ if there exists a cube $Q_s^t$ in $P_m^\ell$ which verifies
	\begin{equation}\label{eq: desigualdad 1 conjunto P_m^l}
	\frac{1}{|Q_j^k|}\int_{Q_j^k} u>\frac{2}{|Q_s^t|}\int_{Q_s^t}u
	\end{equation}
	and it is maximal in this sense, that is,
	\begin{equation}\label{eq: desigualdad 2 conjunto P_m^l}
	\frac{1}{|Q_{j'}^{k'}|}\int_{Q_j^k} u\leq\frac{2}{|Q_s^t|}\int_{Q_s^t}u
	\end{equation}
	for every $Q_j^k\subsetneq Q_{j'}^{k'}\subsetneq Q_s^t$.
	
	Let $P^\ell=\cup_{m\geq 0} P_m^{\ell}$, the set of principal cubes in $\Delta_\ell$. By applying Lemma~\ref{lema: acotacion de u(Ek cap Q_j^k)} and the definition of $\Lambda_{\ell,k}$ we have that
	\begin{align*}
	\sum_{k\geq N}\sum_{\ell\geq 0} \sum_{Q_j^k\in \Gamma_{\ell,k}}a^{(k+1)r}u(E_k\cap Q_j^k)&\leq\sum_{k\geq N}\sum_{\ell\geq 0} \sum_{Q_j^k\in \Gamma_{\ell,k}}c_1a^{(k+1)r}e^{-c_2\ell r}u(Q_j^k)\\
	&\leq \sum_{\ell\geq 0}c_1e^{-c_2\ell r}a^{r(1-\ell)}\sum_{k\geq N}\sum_{Q_j^k\in \Gamma_{\ell,k}}\frac{v^r(Q_j^k)}{|Q_j^k|}u(Q_j^k).
	\end{align*}
	
	Let us sort the inner double sum in a more convenient way. We define
	\[\mathcal{A}_{(t,s)}^\ell=\left\{Q_j^k \in \bigcup_{k\geq N} \Gamma_{\ell,k}: Q_j^k\subseteq Q_s^t \textrm{ and } Q_s^t \textrm{ is the smallest cube in }P^\ell \textrm{ that contains it} \right\}.\]
	That is, every $Q_j^k\in \mathcal{A}_{(t,s)}^\ell$ is not a principal cube, unless $Q_j^k=Q_s^t$. Recall that $v^r\in A_\infty$ implies that there exist two positive constants $C$ and $\varepsilon$ verifying
	\begin{equation}\label{eq: condicion Ainfty de v^r}
	\frac{v^r(E)}{v^r(Q)}\leq C\left(\frac{|E|}{|Q|}\right)^{\varepsilon},
	\end{equation}
	for every cube $Q$ and every measurable set $E$ of $Q$.
	
	By using \eqref{eq: desigualdad 2 conjunto P_m^l} and Lemma~\ref{lema: sparsity de los cubos} we have that
	\begin{align*}
	\sum_{k\geq N}\sum_{Q_j^k\in \Gamma_{\ell,k}}\frac{v^r(Q_j^k)}{|Q_j^k|}u(Q_j^k)&=\sum_{Q_s^t \in P^\ell}\,\,\sum_{(k,j): Q_j^k\in \mathcal{A}_{(t,s)}^\ell}\frac{u(Q_j^k)}{|Q_j^k|}v^r(Q_j^k)\\
	&\leq 2\sum_{Q_s^t \in P^\ell}\frac{u(Q_s^t)}{|Q_s^t|}\,\,\sum_{(k,j): Q_j^k\in \mathcal{A}_{(t,s)}^\ell}v^r(Q_j^k)\\
	&\leq C\sum_{Q_s^t \in P^\ell}\frac{u(Q_s^t)}{|Q_s^t|}v^r(Q_s^t)\left(\frac{\left|\bigcup_{(k,j): Q_j^k\in \mathcal{A}_{(t,s)}^\ell}Q_j^k\right|}{|Q_s^t|}\right)^\varepsilon\\
	&\leq C\sum_{Q_s^t \in P^\ell} \frac{u(Q_s^t)}{|Q_s^t|}v^r(Q_s^t).
	\end{align*}
	
	Therefore,
	\begin{align*}
	\sum_{k\geq N}\sum_{\ell\geq 0} \sum_{Q_j^k\in \Gamma_{\ell,k}}a^{(k+1)r}u(E_k\cap Q_j^k)&\leq C\sum_{\ell\geq 0}e^{-c_2\ell r}a^{-\ell r}\sum_{Q_s^t \in P^\ell} \frac{v^r(Q_s^t)}{|Q_s^t|}u(Q_s^t)\\
	&\leq C\sum_{\ell\geq 0}e^{-c_2\ell r}\sum_{Q_s^t \in P^\ell} a^{tr}u(Q_s^t).
	\end{align*}
	
	\begin{afirmacion}\label{af: control de a^kr por promedios de Phi(f), caso l no negativo}
		Given $\ell\geq 0$ and $Q_j^k\in \bigcup_{k\geq N} \Gamma_{\ell,k}$, there exists a positive constant $C$, independent of $\ell$, such that
		\begin{equation}\label{eq: estimacion de afirmacion: control de akr por promedios de Phi(f). caso l no negativo}
		a^{kr}\leq \frac{C}{|Q_j^k|}\int_{Q_j^k}\Phi\left(|f(x)|\right)v^r(x)\,dx.
		\end{equation}
	\end{afirmacion}
	
	Using this claim, we obtain
	\begin{align*}
	\sum_{k\geq N}\sum_{\ell\geq 0} \sum_{Q_j^k\in \Gamma_{\ell,k}}a^{(k+1)r}u(E_k\cap Q_j^k)&\leq C\sum_{\ell\geq 0} e^{-c_2\ell r}\sum_{Q_s^t\in P^{\ell}}\frac{u(Q_s^t)}{|Q_s^t|}\int_{Q_s^t}\Phi\left(|f|\right)v^r\\
	&=C\sum_{\ell\geq 0}e^{-c_2\ell r}\int_{\mathbb{R}^n} \Phi\left(|f(x)|\right)v^r(x)\left(\sum_{Q_s^t\in P^{\ell}}\frac{u(Q_s^t)}{|Q_s^t|}\mathcal{X}_{Q_s^t}(x)\right)\,dx\\
	&=C\sum_{\ell\geq 0}e^{-c_2\ell r}\int_{\mathbb{R}^n} \Phi\left(|f(x)|\right)v^r(x)h_1(x)\,dx
	\end{align*}
	
	\begin{afirmacion}\label{af: control de h_1 por u}
	There exists a positive constant $C$, independent of $\ell$, that satisfies $h_1(x)\leq Cu(x)$.
	\end{afirmacion}
	With this claim at hand, we can obtain 
	\[A_N\leq C \int_{\mathbb{R}^n}\Phi\left(|f(x)|\right)u(x)v^r(x)\,dx,\]
	where $C$ does not depend on $N$.
	
	Let us center our attention on the estimate of $B_N$. Fix $0<\beta<\varepsilon$, where $\varepsilon$ is the number appearing in \eqref{eq: condicion Ainfty de v^r}. We shall build the set of principal cubes in $\Delta_{-1}=\bigcup_{k\geq N}\Gamma_{-1,k}$. Let
	\[P_0^{-1}=\{Q: Q \textrm{ is a maximal cube in }\Delta_{-1}\textrm{ in the sense of inclusion}\}\]
	and, recursively, we say that $Q_{j,i}^k\in P_{m+1}^{-1}$, $m\geq 0$, if there exists a cube $Q_{s,l}^t\in P_m^{-1}$ such that
	\begin{equation}\label{eq: desigualdad 1 conjunto P_m^{-1}}
	\frac{1}{|Q_{j,i}^k|}\int_{Q_{j,i}^k} u> \frac{a^{(k-t)\beta r}}{|Q_{s,l}^t|}\int_{Q_{s,l}^t}u
	\end{equation} 
	and it is the biggest subcube of $Q_{s,l}^t$ that verifies this condition, that is
	\begin{equation}\label{eq: desigualdad 2 conjunto P_m^{-1}}
	\frac{1}{|Q_{j',i'}^{k'}|}\int_{Q_{j',i'}^{k'}} u\leq \frac{a^{(k-t)\beta r}}{|Q_{s,l}^t|}\int_{Q_{s,l}^t}u
	\end{equation} 
	if $Q_{j,i}^k\subsetneq Q_{j',i'}^{k'}\subsetneq Q_{s,l}^t$. Let $P^{-1}=\bigcup_{m\geq 0} P_m^{-1}$, the set of principal cubes in $\Delta_{-1}$. Similarly as before, we define the set
	\[\mathcal{A}_{(t,s,l)}^{-1}=\left\{Q_{j,i}^k \in \bigcup_{k\geq N} \Gamma_{-1,k}: Q_{j,i}^k\subseteq Q_{s,l}^t \textrm{ and } Q_{s,l}^t \textrm{ is the smallest cube in }P^{-1} \textrm{ that contains it} \right\}.\]
	
	We can therefore estimate $B_N$ as follows
	\begin{align*}
	B_N&\leq a^r \sum_{k\geq N}\sum_{i:Q_{j,i}^k\in \Gamma_{-1,k}}\frac{v^r(Q_{j,i}^k)}{|Q_{j,i}^k|}u(Q_{j,i}^k)\\
	&\leq a^r\sum_{Q_{s.l}^t \in P^{-1}}\sum_{k,j,i: Q_{j,i}^k\in \mathcal{A}_{(t,s,l)}^{-1}}\frac{u(Q_{j,i}^k)}{|Q_{j,i}^k|}v^r(Q_{j,i}^k)\\
	&\leq a^r\sum_{Q_{s.l}^t \in P^{-1}}\frac{u(Q_{s,l}^t)}{|Q_{s,l}^t|}\sum_{k\geq t}a^{(k-t)\beta r}\,\sum_{j,i: Q_{j,i}^k\in \mathcal{A}_{(t,s,l)}^{-1}}v^r(Q_{j,i}^k).
	\end{align*}
	
	Fixed $k\geq t$, observe that
	\[\sum_{j,i: Q_{j,i}^k\in \mathcal{A}_{(t,s,l)}^{-1}} |Q_{j,i}^k|<\sum_{j,i: Q_{j,i}^k\in \mathcal{A}_{(t,s,l)}^{-1}} a^{-kr}v^r(Q_{j,i}^k)\leq a^{-kr}v^r(Q_{s,l}^t)\leq 2^na^{(t-k)r}|Q_{s,l}^t|.\]
	Combining this inequality with the $A_\infty$ condition of $v^r$ we have, for every $k\geq t$, that
	\begin{align*}
	\sum_{j,i: Q_{j,i}^k\in \mathcal{A}_{(t,s,l)}^{-1}}a^{(k-t)\beta r}v^r(Q_{j,i}^k)&\leq Cv^r(Q_{s,l}^t)\left(\frac{\sum_{j,i: Q_{j,i}^k\in \mathcal{A}_{(t,s,l)}^{-1}}|Q_{j,i}^k|}{|Q_{s,l}^t|}\right)^\varepsilon\\
	&\leq Ca^{(t-k)r\varepsilon}.
	\end{align*}
	Thus,
	\begin{align*}
	B_N&\leq C\sum_{Q_{s.l}^t \in P^{-1}}\frac{u(Q_{s,l}^t)}{|Q_{s,l}^t|}v^r(Q_{s,l}^t)\sum_{k\geq t} a^{(t-k)r(\varepsilon-\beta)}\\
	&=C\sum_{Q_{s.l}^t \in P^{-1}}\frac{v^r(Q_{s,l}^t)}{|Q_{s,l}^t|}u(Q_{s,l}^t)\\
	&\leq C\sum_{Q_{s.l}^t \in P^{-1}}a^{tr}u(Q_{s,l}^t).
	\end{align*}
	
	\begin{afirmacion}\label{af: control de a^kr por promedios de Phi(f), caso l=-1}
		If $Q_{j}^k\in \Lambda_{-1,k}$ then there exists a positive constant $C$ such that
		\[a^{kr}\leq \frac{C}{|Q_j^k|}\int_{Q_j^k}\Phi\left(|f(x)|\right)v^r(x)\,dx.\]
	\end{afirmacion}
	We can proceed now as follows
	\begin{align*}
	\sum_{k\geq N}\sum_{i:Q_{j,i}^k\in \Gamma_{-1,k}} a^{(k+1)r}u(Q_{j,i}^k)&\leq C\sum_{Q_{s.l}^t \in P^{-1}}a^{tr}u(Q_{s,l}^t)\\
	&\leq C\sum_{Q_{s.l}^t \in P^{-1}}\frac{u(Q_{s,l}^t)}{|Q_s^t|}\int_{Q_s^t}\Phi\left(|f(x)|\right)v^r(x)\,dx\\
	&\leq C\int_{\mathbb{R}^n} \Phi\left(|f(x)|\right)v^r(x)\left[\sum_{Q_{s.l}^t \in P^{-1}}\frac{u(Q_{s,l}^t)}{|Q_s^t|}\mathcal{X}_{Q_s^t}(x)\right]\,dx\\
	&=C\int_{\mathbb{R}^n} \Phi\left(|f(x)|\right)v^r(x)h_2(x)\,dx.
	\end{align*}
	
	\begin{afirmacion}\label{af: control de h_2 por u}
		There exists a positive constant $C$, independent of $N$, that verifies $h_2(x)\leq Cu(x)$, for almost every $x$.
	\end{afirmacion}
	
	This claim allows to obtain the desired estimate for $B_N$. This completes the proof.\qedhere
\end{proof}

\medskip

In order to conclude, we give the proof of the claims. 

\medskip

\begin{proof}[Proof of Claim~\ref{af: control de a^kr por promedios de Phi(f), caso l no negativo}]
Fix $\ell\geq 0$ and a cube $Q_j^k\in \bigcup_{k\geq N} \Gamma_{\ell,k}$. We know that $\|g\|_{\Phi,Q_j^k}>a^k$ or,  equivalently, $\left\|\frac{g}{a^k}\right\|_{\Phi, Q_j^k}>1$. Denote with $A=\{x\in Q_j^k: v(x)\leq t^*a^k\}$ and $B=Q_j^k\backslash A$, where $t^*$ is the number verifying that if $z\geq t^*$, then
\[\frac{\Phi(z)}{z^r}\leq C_0 \left(\log z\right)^\delta.\]
Then,
\[1<\left\|\frac{g}{a^k}\right\|_{\Phi, Q_j^k}\leq \left\|\frac{g}{a^k}\mathcal{X}_A\right\|_{\Phi, Q_j^k}+\left\|\frac{g}{a^k}\mathcal{X}_B\right\|_{\Phi, Q_j^k}=I+II.\]
This inequality implies that either $I>1/2$ or $II>1/2$. If the first case holds, since $\Phi\in \mathfrak{F}_r$ we have that
\begin{align*}
1&<\frac{1}{|Q_j^k|}\int_A\Phi\left(\frac{2|f|v}{a^k}\right)\\
&\leq \frac{\Phi(2t^*)}{|Q_j^k|}\int_A \Phi\left(|f|\right)\left(\frac{v}{t^* a^k}\right)^r\\
&\leq \frac{C}{a^{kr}}\frac{1}{|Q_j^k|}\int_{Q_j^k}\Phi\left(|f|\right)v^r,
\end{align*}
and from here we can obtain
\[a^{kr}<\frac{C}{|Q_j^k|}\int_{Q_j^k}\Phi\left(|f|\right)v^r.\]
On the other hand, if $II>1/2$ then again 
\begin{align*}
1&<\frac{1}{|Q_j^k|}\int_B\Phi\left(\frac{2|f|v}{a^k}\right)\\
&\leq \Phi(2)C_0\frac{1}{|Q_j^k|}\int_B \Phi\left(|f|\right)\frac{v^r}{a^{kr}}\left(\log\left(\frac{v}{a^k}\right)\right)^{\delta}\mathcal{X}_{B},
\end{align*}
since $\Phi\in \mathfrak{F}_r$. This implies that
\[a^{kr}\leq \frac{\Phi(2)C_0}{|Q_j^k|}\int_{Q_j^k}\Phi\left(|f|\right)v^rw_k,\]
where $w_k(x)=\left(\log\left(\frac{v(x)}{a^k}\right)\right)^{\delta}\mathcal{X}_{B}(x)$. 

Since $v^r\in A_\infty$, there exists $s>1$ such that $v^r\in \textrm{RH}_s$. Let $\delta_0=\max\{\delta/r,1\}$ and fix $0<\varepsilon<\varepsilon_0$, where
\[\varepsilon_0\leq\min\left\{\frac{1}{\delta_0 s'a^{(\ell+1)r/s'}[v^r]_{\textrm{RH}_s}-1}, \frac{\log 2}{\log (2\Phi(2)C_0e^{2/e}a^{\ell r})}\right\}.\]
Now we define $\gamma=1+\varepsilon$, therefore $\gamma'=1+1/\varepsilon$. By applying H\"{o}lder's inequality with $\gamma$ and $\gamma'$ with respect to the measure $d\mu(x)=v^r(x)\,dx$, we get
\begin{equation}\label{eq: afirmacion - control de a^kr por promedios de Phi(f), caso l no negativo - eq1}
a^{kr}<\Phi(2)C_0\left(\frac{v^r(Q_j^k)}{|Q_j^k|}\right)\left(\frac{1}{v^r(Q_j^k)}\int_{Q_j^k} \left[\Phi\left(|f|\right)\right]^{\gamma}v^r\right)^{1/\gamma}\left(\frac{1}{v^r(Q_j^k)}\int_{Q_j^k}w_k^{\gamma'}v^r\right)^{1/\gamma'}.
\end{equation}
Let us analyze the third factor. From the well-known fact that $\log t \leq \xi^{-1}t^{\xi}$ for every $t,\xi>0$ and applying H\"{o}lder's inequality with $s$ and $s'$ we have
\begin{align*}
\left(\frac{1}{v^r(Q_j^k)}\int_{Q_j^k}w_k^{\gamma'}v^r\right)^{1/\gamma'}&=\left(\frac{1}{v^r(Q_j^k)}\int_{Q_j^k}\left(\log\left(\frac{v}{a^k}\right)\right)^{\delta\gamma'}v^r\right)^{1/\gamma'}\\
&\leq \left(\frac{1}{v^r(Q_j^k)}\int_{Q_j^k\cap B}\frac{\delta s'\gamma'}{r}\left(\frac{v}{a^k}\right)^{r/s'}v^r\right)^{1/\gamma'}\\
&\leq \left(\delta_0 s' \gamma'\frac{|Q_j^k|}{v^r(Q_j^k)}\right)^{1/\gamma'}\left[\left(\frac{1}{|Q_j^k|}\int_{Q_j^k}\frac{v^r}{a^{kr}}\right)^{1/s'}\left(\frac{1}{|Q_j^k|}\int_{Q_j^k}v^{rs}\right)^{1/s}\right]^{1/\gamma'}\\
&\leq \left[\delta_0 s'\gamma' \frac{|Q_j^k|}{v^r(Q_j^k)} a^{(\ell+1)r/s'}[v^r]_{\textrm{RH}_s} \frac{v^r(Q_j^k)}{|Q_j^k|}\right]^{1/\gamma'}\\
&=\left[\delta_0 s'\gamma' a^{(\ell+1)r/s'}[v^r]_{\textrm{RH}_s} \right]^{1/\gamma'}\\
&\leq \left(\gamma'\right)^{2/\gamma'}\\
&\leq e^{2/e},
\end{align*}
by virtue of the election for $\varepsilon$ and Lemma~\ref{lema: estimacion constante epsilon}.


Returning to \eqref{eq: afirmacion - control de a^kr por promedios de Phi(f), caso l no negativo - eq1} we have that
\[a^{kr}\leq D\left(\frac{v^r(Q_j^k)}{|Q_j^k|}\right)\left(\frac{1}{v^r(Q_j^k)}\int_{Q_j^k} \left[\Phi\left(|f|\right)\right]^{\gamma}v^r\right)^{1/\gamma},\]
where $D=\Phi(2)C_0e^{2/e}$.
By denoting $\Psi(t)=t^{\gamma}$, we have that the second factor is $\|\Phi(f)\|_{\Psi, v^r, Q_j^k}$. Using \eqref{eq: preliminares - equivalencia normal Luxemburo con infimo} we have that for every $\tau>0$
\begin{align*}
a^{kr}&\leq D\frac{v^r(Q_j^k)}{|Q_j^k|}\left\{\tau+\frac{\tau^{1-\gamma}}{v^r(Q_j^k)}\int_{Q_j^k} \left[\Phi\left(|f|\right)\right]^\gamma v^r\right\}\\
&\leq D\tau a^{(k+\ell)r}+D\frac{\tau^{1-\gamma}}{|Q_j^k|}\int_{Q_j^k} \left[\Phi\left(|f|\right)\right]^\gamma v^r.
\end{align*}
Pick $\tau=1/(2Da^{\ell r})$ and observe that with this choice
\[D\tau a^{(k+\ell)r}=\frac{a^{kr}}{2},\quad \textrm{ and }\quad \tau^{1-\gamma}=(2Da^{\ell r})^{\varepsilon}\leq 2,\]
by virtue of the definition of $\varepsilon$. Thus,
\[a^{kr}\leq \frac{4D}{|Q_j^k|}\int_{Q_j^k}\left[\Phi\left(|f|\right)\right]^\gamma v^r,\]
for every $0<\varepsilon<\varepsilon_0$, with $D$ independent of $\varepsilon$. The dominate convergence theorem allows to conclude the thesis by letting $\varepsilon\to 0$. \qedhere
\end{proof}

\medskip

\begin{proof}[Proof of Claim~\ref{af: control de h_1 por u}]
Fix $\ell\geq 0$ and $x\in \mathbb{R}^n$ such that $u(x)<\infty$. We shall consider a sequence of nested principal cubes in $\Delta_\ell$ that contain $x$. Let $Q^{(0)}$ be the maximal cube (in the sense of inclusion) in $P^{\ell}$ that contains $x$. In general, given $Q^{(j)}$ we denote with $Q^{(j+1)}$ the maximal principal cube in $Q^{(j)}$ that contains $x$. This so-defined sequence has only a finite number of terms. If not, for every $j$ we would have
\[\frac{1}{|Q^{(0)}|}\int_{Q^{(0)}}u\leq \frac{1}{2^{j}}\frac{1}{|Q^{(j)}|}\int_{Q^{(j)}}u\leq \frac{[u]_{A_1}}{2^j}u(x),\]
or equivalently
\[\frac{2^j}{|Q^{(0)}|}\int_{Q^{(0)}}u\leq [u]_{A_1}u(x),\]
and we would get a contradiction by letting $j\to \infty$. Therefore, $x$ can only belong to a finite number $J=J(x)$ of these cubes. Thus,
\begin{align*}
\sum_{Q_s^t\in P^{\ell}} \frac{u(Q_s^t)}{|Q_s^t|}\mathcal{X}_{Q_s^t}(x)&\leq \sum_{j=0}^J \frac{1}{|Q^{(j)}|}\int_{Q^{(j)}}u\\
&\leq \sum_{j=0}^J 2^{j-J}\frac{1}{|Q^{(J)}|}\int_{Q^{(J)}}u\\
&\leq [u]_{A_1}u(x) \sum_{j=0}^J 2^{j-J}\\
&\leq [u]_{A_1}u(x) 2^{-J}(2^{J+1}-1)\\
&\leq 2[u]_{A_1}u(x).\qedhere
\end{align*}
\end{proof}

\medskip

\begin{proof}[Proof of Claim~\ref{af: control de a^kr por promedios de Phi(f), caso l=-1}]
This proof is similar to the given for Claim~\ref{af: control de a^kr por promedios de Phi(f), caso l no negativo}, with some obvious changes since the average of $v^r$ over $Q_j^k$ is not equivalent to $a^{(\ell+k)r}$. Following the same notation as in Claim~\ref{af: control de a^kr por promedios de Phi(f), caso l no negativo}, since $\left\|\frac{g}{a^k}\right\|_{\Phi, Q_j^k}>1$, we have that either $I>1/2$ or $II>1/2$. If $I>1/2$, we obtain the thesis exactly in the same way as in this claim. On the other hand, if $II>1/2$, we have that
\[a^{kr}\leq \frac{C}{|Q_j^k|}\int_{Q_j^k}\Phi\left(|f|\right)v^rw_k,\]
where $w_k=\left(\log\left(\frac{v}{a^k}\right)\right)^{\delta}\mathcal{X}_{B}$. 

Fix $0<\varepsilon<\varepsilon_0$, where 
\[\varepsilon_0\leq \frac{1}{[v^r]_{\textrm{RH}_s}\delta_0 s'-1}\]
and set $\gamma=1+\varepsilon$. We apply H\"{o}lder's inequality with $\gamma$ and $\gamma'$ with respect to $v^r$ to obtain
\begin{equation}\label{eq: afirmacion - control de a^kr por promedios de Phi(f), caso l=-1 - eq1}
a^{kr}\leq C\frac{v^r(Q_j^k)}{|Q_j^k|}\left(\frac{1}{v^r(Q_j^k)}\int_{Q_j^k}\left[\Phi\left(|f|\right)\right]^\gamma v^r\right)^{1/\gamma}\left(\frac{1}{v^r(Q_j^k)}\int_{Q_j^k}\left(\log\left(\frac{v}{a^k}\right)\right)^{\delta \gamma'} v^r\mathcal{X}_{B}\right)^{1/\gamma'}
\end{equation}	
Recall that $Q_j^k$ satisfies $|Q_j^k|^{-1}\int_{Q_j^k} v^r<a^{kr}$. Thus, we can estimate the third factor as follows
\begin{align*}
\left(\frac{1}{v^r(Q_j^k)}\int_{B}\left(\log\left(\frac{v}{a^k}\right)\right)^{\delta \gamma'} v^r\right)^{1/\gamma'}&\leq \left(\frac{1}{v^r(Q_j^k)}\int_{B}\frac{\delta s'\gamma'}{r}\left(\frac{v}{a^k}\right)^{r/s'}v^r\right)^{1/\gamma'}\\
&\leq \left[\frac{\delta_0 s' \gamma'|Q_j^k|}{v^r(Q_j^k)}\left(\frac{1}{|Q_j^k|}\int_{Q_j^k}\frac{v^r}{a^{kr}}\right)^{1/s'}\left(\frac{1}{|Q_j^k|}\int_{Q_j^k}v^{rs}\right)^{1/s}\right]^{1/\gamma'}\\
&\leq \left([v^r]_{\textrm{RH}_s}\delta_0 s' \gamma'\right)^{1/\gamma'}\\
&\leq \left(\gamma'\right)^{2/\gamma'}\\
&\leq e^{2/e},
\end{align*}
from the choice for $\varepsilon$ and Lemma~\ref{lema: estimacion constante epsilon}. Returning to \eqref{eq: afirmacion - control de a^kr por promedios de Phi(f), caso l=-1 - eq1} we obtain
\[a^{kr}\leq D\frac{v^r(Q_j^k)}{|Q_j^k|}\left(\frac{1}{v^r(Q_j^k)}\int_{Q_j^k}\left[\Phi\left(|f|\right)\right]^\gamma v^r\right)^{1/\gamma},\]
with $D=Ce^{2/e}$. Similarly as we did in the proof of Claim~\ref{af: control de a^kr por promedios de Phi(f), caso l no negativo} we can conclude that
\begin{align*}
a^{kr}&\leq D\frac{v^r(Q_j^k)}{|Q_j^k|}\left\{\tau+\frac{\tau^{1-\gamma}}{v^r(Q_j^k)}\int_{Q_j^k} \left[\Phi\left(|f|\right)\right]^\gamma v^r\right\}\\
&\leq D\tau a^{kr}+D\frac{\tau^{1-\gamma}}{|Q_j^k|}\int_{Q_j^k} \left[\Phi\left(|f|\right)\right]^\gamma v^r,
\end{align*}
for every $\tau>0$. Picking $\tau=1/(2D)$ we get
\[D\tau a^{kr}=\frac{a^{kr}}{2},\quad\textrm{ y }\quad \tau^{1-\gamma}=(2D)^{\varepsilon}\leq 2D.\]
Therefore,
\[a^{kr}\leq \frac{4D^2}{|Q_j^k|}\int_{Q_j^k}\left[\Phi\left(|f|\right)\right]^\gamma v^r,\]
for every $0<\varepsilon<\varepsilon_0$. Again, the constant $D$ does not depend on $\varepsilon$. Letting $\varepsilon\to 0$ we obtain the thesis.
\end{proof}

\medskip

\begin{proof}[Proof of Claim~\ref{af: control de h_2 por u}]
Let us fix $x\in \mathbb{R}^n$ and assume that $u(x)<\infty$. For every level $t$, there exists at most one cube $Q_s^t$ such that $x\in Q_s^t$. If this cube does exist, we denoted it by $Q^t$. Let $G=\{t: x\in Q^t\}$. Since $t\geq N$, $G$ is bounded from below. Then there exists $t_0$, the minimum of $G$. We shall build a sequence of elements in $G$ recursively: having chosen $t_m$, with $m\geq 0$, we pick $t_{m+1}$ as the smallest element in $G$ greater than $t_m$ and that verifies
\begin{equation}\label{eq: afirmacion - control de h_2 por u - eq1}
\frac{1}{|Q^{t_{m+1}}|}\int_{Q^{t_{m+1}}}u> \frac{2}{|Q^{t_m}|}\int_{Q^{t_m}}u.
\end{equation}
Observe that if $t\in G$ y $t_m\leq t< t_{m+1}$, then
\begin{equation}\label{eq: afirmacion - control de h_2 por u - eq2}
\frac{1}{|Q^{t}|}\int_{Q^{t}}u\leq \frac{2}{|Q^{t_m}|}\int_{Q^{t_m}}u.
\end{equation}
This sequence has only a finite number of terms. Indeed, if it was not the case, we would have
\[[u]_{A_1}u(x)\geq \frac{1}{|Q^{t_m}|}\int_{Q^{t_m}}u>\frac{2^m}{|Q^{t_0}|}\int_{Q^{t_0}}u\]
for every $m\geq 0$. By letting $m\to \infty$ we would arrive to a contradiction. Then $\{t_m\}=\{t_m\}_{m=0}^{M}$.
Denoting $\mathcal{F}_m=\{t\in G: t_m\leq t< t_{m+1}\}$, and using \eqref{eq: afirmacion - control de h_2 por u - eq2} we can write
\[h_2(x)=\sum_{Q_{s,l}^t\in P^{-1}} \frac{u(Q_{s,l}^t)}{|Q_s^t|}\mathcal{X}_{Q_s^t}(x)\leq \sum_{m=0}^M\left( \frac{2}{|Q^{t_m}|}\int_{Q^{t_m}}u\right)\sum_{t\in \mathcal{F}_m}\,\sum_{s,l: Q_{s,l}^t\in P^{-1}}\frac{u(Q_{s,l}^t)}{u(Q^t)}.\]
We shall prove that there exists a positive constant $C$, independent of $m$, such that
\begin{equation}\label{eq: afirmacion - control de h_2 por u - eq3}
\sum_{t\in \mathcal{F}_m}\,\sum_{s,l: Q_{s,l}^t\in P^{-1}}\frac{u(Q_{s,l}^t)}{u(Q^t)}\leq C.
\end{equation}
If this inequality holds, we get that
\begin{align*}
h_2(x)&\leq 2C\sum_{m=0}^M \frac{1}{|Q^{t_m}|}\int_{Q^{t_m}}u\\
&\leq 2C\sum_{m=0}^M 2^{m-M}\frac{1}{|Q^{t_M}|}\int_{Q^{t_M}}u\\
&\leq 2C[u]_{A_1}u(x)2^{-M}\sum_{m=0}^M 2^{m}\\
&\leq 4C[u]_{A_1}u(x),
\end{align*}
which completes the proof of the claim. To finish, let us prove \eqref{eq: afirmacion - control de h_2 por u - eq3}.

Fix $0\leq m\leq M$ and observe that if $t=t_0$, then
\[\sum_{s,l: Q_{s,l}^t\in P^{-1}}\frac{u(Q_{s,l}^t)}{u(Q^t)}\leq 1.\]
Let $t_m<t<t_{m+1}$. If $Q_{j,i}^{t_m}\cap Q_{s,l}^t\neq \emptyset$, then we must have $Q_{s,l}^t\subsetneq Q_{j,i}^{t_m}$,  otherwise we would have $k_m>t$, a contradiction. Let $Q_{j',i'}^{t'}$ the smallest principal cube that contains $Q_{j,i}^{t_m}$ (this cube does exist because we are assuming $t>t_0$). By applying \eqref{eq: desigualdad 1 conjunto P_m^{-1}} and \eqref{eq: desigualdad 2 conjunto P_m^{-1}} we conclude that
\[\frac{1}{|Q_{s,l}^t|}\int_{Q_{s,l}^t}u>\frac{a^{(t-t')r\beta}}{|Q_{j',i'}^{t'}|}\int_{Q_{j',i'}^{t'}}u\]
and also
\[\frac{1}{|Q_{j,i}^{t_m}|}\int_{Q_{j,i}^{t_m}}u\leq \frac{a^{(t_m-t')r\beta}}{|Q_{j',i'}^{t'}|}\int_{Q_{j',i'}^{t'}}u.\]
Combining these two estimates with \eqref{eq: afirmacion - control de h_2 por u - eq2} we obtain that, for almost every $y\in Q_{s,l}^t$ 
\begin{align*}
u(y)[u]_{A_1}&\geq \frac{1}{|Q_{s,l}^t|}\int_{Q_{s,l}^t}u>\frac{a^{(t-t_m)r\beta}}{|Q_{j,i}^{t_m}|}\int_{Q_{j,i}^{t_m}}u\\
&\geq a^{(t-t_m)r\beta}\inf_{Q_{j,i}^{t_m}}u \geq a^{(t-t_m)r\beta}\inf_{Q^{t_m}}u\\
&\geq \frac{a^{(t-t_m)r\beta}}{[u]_{A_1}}\frac{1}{|Q^{t_m}|}\int_{Q^{t_m}}u\\
&\geq \frac{a^{(t-t_m)r\beta}}{2[u]_{A_1}}\frac{1}{|Q^{t}|}\int_{Q^{t}}u.
\end{align*}  
Then
\[u(y)>\frac{a^{(t-t_m)r\beta}}{2[u]_{A_1}^2}\frac{1}{|Q^{t}|}\int_{Q^{t}}u=:\lambda.\]
Since $u\in A_1\subseteq A_\infty$, there exist two positive constants $C$ and $\nu$ for which the inequality
\[\frac{u(E)}{u(Q)}\leq C\left(\frac{|E|}{|Q|}\right)^\nu,\]
holds for every cube $Q$ and every measurable subset $E$ of $Q$. Therefore,
\begin{align*}
\sum_{s,l: Q_{s,l}^t\in P^{-1}}u(Q_{s,l}^t)&\leq \frac{u(\{y\in Q^t: u(y)>\lambda\})}{u(Q^t)}u(Q^t)\\
&\leq C\left(\frac{|\{y\in Q^t: u(y)>\lambda\}|}{|Q^t|}\right)^\nu u(Q^t)\\
&\leq Cu(Q^t)\left(\frac{1}{\lambda |Q^t|}\int_{Q^t}u\right)^\nu\\
&=C\left(2[u]_{A_1}^2a^{(t_m-t)r\beta}\right)^\nu u(Q^t).
\end{align*}

If $t=t_m$, we have $Q_{s,l}^t=Q_{j,i}^{t_m}$ and in this case
\begin{align*}
u(y)[u]_{A_1}&\geq \frac{1}{|Q_{s,l}^t|}\int_{Q_{s,l}^t}u\geq \inf_{Q_{j,i}^{t_m}}u\\ 
&\geq \inf_{Q^{t_m}}u\geq \frac{1}{[u]_{A_1}|Q^{t_m}|}\int_{Q^{t_m}}u\\
&\geq \frac{1}{2[u]_{A_1}|Q^{t}|}\int_{Q^{t}}u,
\end{align*}
which is the corresponding estimate obtained above, with $t=t_m$. Thus,
\begin{align*}
\sum_{t\in \mathcal{F}_m}\sum_{s,l: Q_{s,l}^t\in P^{-1}} \frac{u(Q_{s,l}^t)}{u(Q^t)}&\leq \sum_{t\geq t_m} \left(2[u]_{A_1}a^{(t_m-t)r\beta}\right)^\nu\\
&\leq C\sum_{t\geq t_m} a^{(t_m-t)r\beta \nu}\\
&=C,
\end{align*}
which proves \eqref{eq: afirmacion - control de h_2 por u - eq3}. \qedhere
\end{proof}

In order to prove Corollary~\ref{coro: corolario del teorema principal} we need the following result. A proof can be found in \cite{Berra-Carena-Pradolini(MN)}.

\begin{lema}\label{lema - implicacion de tipo debil para interpolacion modular}
	Let $\mu$ be a measure, $T$ a sub-additive operator, and $\varphi$ a Young function. Assume that
	\[\mu(\{x: |Tf(x)|>t\})\leq C\int_{\mathbb{R}^n}\varphi\left(\frac{c|f(x)|}{t}\right)\,d\mu(x),\]
	for some positive constants $C$ and $c$, and every $t>0$. Also assume that $\left\|Tf\right\|_{L^\infty(\mu)}\leq C_0\left\|f\right\|_{L^\infty(\mu)}$. Then
	\[\mu\left(\left\{x: |Tf(x)|>t\right\}\right)\leq C\int_{\{x: |f(x)|>t/(2C_0)\}}\varphi\left(\frac{2c|f(x)|}{t}\right)\,d\mu(x).\]
\end{lema}

\begin{proof}[Proof of Corollary~\ref{coro: corolario del teorema principal}]
	The equivalence between $\Phi$ and $\Psi$ imply that there exist positive constants $A$ and $B$ such that
	\[A\Psi(t)\leq \Phi(t)\leq B\Psi(t),\]
	for $t\geq t^*$. Proposition~\ref{propo: normas de funciones equivalentes para t grande son equivalentes} establishes that there exist two positive constants $D$ and $E$ such that
	\[D M_{\Phi}(fv)(x)\leq M_{\Psi}(fv)(x)\leq EM_{\Phi}(fv)(x),\]
	for almost every $x$.
	By applying Theorem~\ref{teo: teorema principal} and setting $c_1=E\max\{\Phi(t^*)+B,1\}$ we have that
	\begin{align*}
	uv^r\left(\left\{x\in \mathbb{R}^n: \frac{M_\Psi(fv)(x)}{M_\Psi v(x)}>t
	\right\}\right)&\leq uv^r\left(\left\{x\in \mathbb{R}^n: \frac{M_\Phi(fv)(x)}{M_\Phi v(x)}>\frac{t}{c_1}
	\right\}\right)\\
	&\leq C \int_{\mathbb{R}^n}\Phi\left(\frac{c_1|f|}{t}\right)uv^r.
	\end{align*}
	Observe that
	\[\|\mathcal{T}_\Psi f\|_{L^{\infty}}=\left\|\frac{M_\Psi (fv)}{M_\Psi v}\right\|_{L^{\infty}}\leq  \|f\|_{L^{\infty}},\]
	which directly implies $\|\mathcal{T}_\Psi f\|_{L^{\infty}(uv^r)}\leq \|f\|_{L^{\infty}(uv^r)}$ since the measure given by $d\mu(x)=u(x)v^r(x)\,dx$ is absolutely continuous with respect to the Lebesgue measure. We now apply Lemma~\ref{lema - implicacion de tipo debil para interpolacion modular} with $T=\mathcal{T}_\Psi$, $C_0=1$, $\varphi=\Phi$ and $\mu$ the measure given above to obtain
	\begin{align*}
	uv^r\left(\left\{x\in \mathbb{R}^n: \frac{M_\Psi(fv)(x)}{M_\Psi v(x)}>t
	\right\}\right)&\leq C\int_{\{x: |f(x)|>t/2\}}\Phi\left(\frac{2c_1|f(x)|}{t}\right)u(x)v^r(x)\,dx\\
	&\leq C\Phi\left(\frac{c_1}{t^*}\right)\int_{\{x: |f(x)|>t/2\}}\Phi\left(\frac{2t^*|f(x)|}{t}\right)u(x)v^r(x)\,dx\\
	&\leq BC\Phi\left(\frac{c_1}{t^*}\right)\int_{\{x: |f(x)|>t/2\}}\Psi\left(\frac{2t^*|f(x)|}{t}\right)u(x)v^r(x)\,dx\\
	&\leq C_1\int_{\mathbb{R}^n}\Psi\left(\frac{C_2|f(x)|}{t}\right)u(x)v^r(x)\,dx.\qedhere
	\end{align*}
\end{proof}

\section{Applications: Mixed inequalities for the generalized fractional maximal operator}\label{section: aplicaciones}
	
We devote this section to prove theorems~\ref{teo: mixta para M_{gamma,Phi}, caso r<p<n/gamma} and \ref{teo: mixta para M_{gamma,Phi}, caso p=r}.

\begin{proof}[Proof of Theorem~\ref{teo: mixta para M_{gamma,Phi}, caso r<p<n/gamma}]
	Define 
	\[\sigma=\frac{nr}{n-r\gamma}, \quad \nu=\frac{n\delta}{n-r\gamma}, \quad \beta=\frac{q}{\sigma}\left(\frac{1}{p}+\frac{1}{r'}\right),\]
	and let $\xi$ be the auxiliary function given by
	\[\xi(t)=\left\{\begin{array}{ccr}
	t^{q/\beta},&\textrm{ if } & 0\leq t\leq 1,\\
	t^\sigma(1+\log^+t)^\nu, & \textrm{ if } &t> 1.
	\end{array}\right.\]
	By virtue of \eqref{eq: inversa generalizada de Phi} we have that
	\[\xi^{-1}(t)t^{\gamma/n}\approx \frac{t^{1/\sigma+\gamma/n}}{(1+\log^+t)^{\nu/\sigma}}=\frac{t^{1/r}}{(1+\log^+t)^{\delta/r}}\approx \Phi^{-1}(t),\]
	for every $t\geq 1$.
	Observe that $\beta>1$: indeed, since $p>r$ we have $q>\sigma$ and thus $q/(\sigma r')>1/r'$. On the other hand, $q/(p\sigma)>1/r$. By combining these two inequalities we have $\beta>1$. Applying Proposition~\ref{propo: estimacion puntual M_{gamma, Phi}} and Lemma~\ref{lema: Jensen para promedios de tipo Luxemburgo} with $\beta$  we can conclude that
	\begin{equation}\label{eq: teo - mixta para M_{gamma,Phi}, caso r<p<n/gamma - eq1}
	M_{\gamma,\Phi}\left(\frac{f_0}{w}\right)(x)\leq C\left[M_\xi\left(\frac{f_0^{p\beta/q}}{w^{\beta}}\right)(x)\right]^{1/\beta}\left(\int_{\mathbb{R}^n}f_0^p(y)\,dy\right)^{\gamma/n}.
	\end{equation}
	Also observe that 
	\begin{equation}\label{eq: teo - mixta para M_{gamma,Phi}, caso r<p<n/gamma - eq2}
	\left(M_\xi(v^\beta)(x)\right)^{1/\beta}\lesssim M_\eta v(x), \quad \textrm{ a.e. }x.
	\end{equation}
	Indeed, it is clear that $\xi(z^\beta)\lesssim \eta(t)$. Given $x$ and a fixed cube $Q$ containing it, we can write
	\begin{align*}
	\frac{1}{|Q|}\int_Q \xi\left(\frac{v^\beta}{\left\|v\right\|_{\eta,Q}^\beta}\right)&\lesssim\frac{1}{|Q|}\int_Q \eta\left(\frac{v}{\left\|v\right\|_{\eta,Q}}\right)\\
	&\leq 1,
	\end{align*}
	which directly implies the estimate.
	
	Notice that $\xi$ is equivalent to a Young function in $\mathfrak{F}_\sigma$, for $t\geq 1$. Since $q(1/p+1/r')=\beta\sigma$, if we set $f_0=|f|wv$, then we can use inequalities \eqref{eq: teo - mixta para M_{gamma,Phi}, caso r<p<n/gamma - eq1} and \eqref{eq: teo - mixta para M_{gamma,Phi}, caso r<p<n/gamma - eq2} and Corollary~\ref{coro: corolario del teorema principal} to estimate
	\begin{align*}
	uv^{\tfrac{q}{p}+\tfrac{q}{r'}}\left(\left\{x: \frac{M_{\gamma,\Phi}(fv)(x)}{M_\eta v(x)}>t\right\}\right)&\lesssim uv^{\beta\sigma}\left(\left\{x: \frac{M_{\gamma,\Phi}(fv)(x)}{\left(M_\xi v^\beta(x)\right)^{1/\beta}}>t\right\}\right)\\
	&\leq uv^{\beta\sigma}\left(\left\{x: \frac{M_\xi\left(f_0^{p\beta/q}w^{-\beta}\right)(x)}{M_\xi v^\beta(x)}>\frac{t^\beta}{\left(\int|f_0|^p\right)^{\beta\gamma/n}}\right\}\right)\\
	&\leq C_1\int_{\mathbb{R}^n}\xi\left(C_2\frac{|f|^{p\beta/q}(wv)^{\beta(p/q-1)}}{t^\beta}\left[\int_{\mathbb{R}^n}|f|^p(wv)^p\right]^{\gamma/n\beta}\right)uv^{\sigma\beta}\\
	&=C_1\int_{\mathbb{R}^n}\xi(\lambda)uv^{\sigma\beta}\\
	&=C_1\left(\int_{A}\xi(\lambda)uv^{\sigma\beta}+\int_{B}\xi(\lambda)uv^{\sigma\beta}\right),
	\end{align*}
	where
	\[\lambda=C_2\frac{|f|^{p\beta/q}(wv)^{\beta(p/q-1)}}{t^\beta}\left[\int_{\mathbb{R}^n}|f|^p(wv)^p\right]^{\gamma/n\beta},\]
	$A=\{x\in \mathbb{R}^n: \lambda(x)\leq 1\}$ y $B=\mathbb{R}^n\backslash A$. By definition of $\xi$ we have that
	\[\int_{A}\xi(\lambda(x))u(x)[v(x)]^{\sigma\beta}\,dx=\int_A [\lambda(x)]^{q/\beta}u(x)[v(x)]^{\sigma\beta}\,dx.\]
	If we set $w=u^{1/q}v^{1/p+1/r'-1}$, then 
	\begin{align*}
	\lambda^{q/\beta}uv^{\sigma\beta}&=C_2^{q/\beta}\frac{|f|^p}{t^q}(wv)^{p-q}\left[\int_{\mathbb{R}^n}|f|^p(wv)^p\right]^{q\gamma/n}uv^{\sigma\beta}\\
	&=C_2^{q/\beta}\frac{|f|^p}{t^q}\left[\int_{\mathbb{R}^n}|f|^p(wv)^p\right]^{q\gamma/n}u^{p/q}v^{\sigma\beta+(p-q)(1/p+1/r')}.
	\end{align*}
	Observe that
	\[\sigma\beta+(p-q)\left(\frac{1}{p}+\frac{1}{r'}\right)=q\left(\frac{1}{p}+\frac{1}{r'}\right)+(p-q)\left(\frac{1}{p}+\frac{1}{r'}\right)=1+\frac{p}{r'}.\]
	Also, notice that
	\[(wv)^p=u^{p/q}v^{1+p/r'-p+p}=u^{p/q}v^{1+p/r'}.\]
	Therefore,
	\begin{align*}
	\int_{A}\xi(\lambda)uv^{\sigma\beta}&\leq \frac{C_2^{q/\beta}}{t^q}\left[\int_{\mathbb{R}^n}|f|^pu^{p/q}v^{1+p/r'}\right]^{q\gamma/n}\left[\int_{\mathbb{R}^n} |f|^pu^{p/q}v^{1+p/r'}\right]\\
	&= \frac{C_2^{q/\beta}}{t^q}\left[\int_{\mathbb{R}^n}|f|^pu^{p/q}v^{1+p/r'}\right]^{1+q\gamma/n}\\
	&=\frac{C_2^{q/\beta}}{t^q}\left[\int_{\mathbb{R}^n}|f|^pu^{p/q}v^{1+p/r'}\right]^{q/p}.
	\end{align*}
	On the other hand, $\lambda(x)>1$ over $B$ and since $\xi$ has an upper type $q/\beta$, we can estimate the integrand by $\lambda^{q/\beta}uv^{\sigma\beta}$. Then we can conclude the estimate by proceeeding like we did in part $A$.
	Thus, we obtain 
	\[uv^{q(1/p+1/r')}\left(\left\{x\in \mathbb{R}^n: \frac{M_{\gamma,\Phi}(fv)(x)}{M_\eta v(x)}>t\right\}\right)^{1/q}\leq C\left[ \int_{\mathbb{R}^n}\left(\frac{|f|}{t}\right)^pu^{p/q}v^{1+p/r'}\right]^{1/p}.\qedhere\]
\end{proof}

\medskip

\begin{proof}[Proof of Theorem~\ref{teo: mixta para M_{gamma,Phi}, caso p=r}]
	Set $\xi(t)=t^q(1+\log^+t)^\nu$, where $\nu=\delta q/r$. Thus $t^{\gamma/n}\xi^{-1}(t)\lesssim \Phi^{-1}(t)$. By applying Proposition~\ref{propo: estimacion puntual M_{gamma, Phi}} with $p=r$ we have that
	\[M_{\gamma,\Phi}\left(\frac{f_0}{w}\right)(x)\leq C\left[M_\xi\left(\frac{f_0^{r/q}}{w}\right)\right](x)\left(\int_{\mathbb{R}^n}f_0^r(y)\,dy\right)^{\gamma/n}.\]
	Observe that in this case we have $\xi =\eta$.
	By setting $f_0=|f|wv$ we can write
	\begin{align*}
	uv^{q}\left(\left\{x: \frac{M_{\gamma,\Phi}(fv)(x)}{M_\eta v(x)}>t\right\}\right)&=uv^{q}\left(\left\{x: \frac{M_{\gamma,\Phi}(f_0/w)(x)}{M_\eta v(x)}>t\right\}\right)\\
	&\leq uv^{q}\left(\left\{x: \frac{M_{\xi}(f_0^{r/q}/w)(x)}{M_\xi v(x)}>\frac{t}{\left(\int f_0^r \right)^{\gamma/n}}\right\}\right).
	\end{align*}
	Since $\xi\in \mathfrak{F}_q$,  we can use the mixed estimate for $M_\xi$ which leads us to
	\begin{equation}\label{eq: eq1 - teo mixta para M_{gamma,Phi}, caso p=r}
	uv^{q}\left(\left\{x: \frac{M_{\gamma,\Phi}(fv)(x)}{M_\eta v(x)}>t\right\}\right)\leq C\int_{\mathbb{R}^n}\xi\left(\frac{f_0^{r/q}\left(\int f_0^r\right)^{\gamma/n}}{wv t}\right)uv^q.
	\end{equation}
	The argument of $\xi$ above can be written as
	
	\begin{align*}
	\frac{f_0^{r/q}\left(\int f_0^r\right)^{\gamma/n}}{wv t}&=\left(\frac{|f|}{t}\right)^{r/q}(wv)^{r/q-1}\left(\int_{\mathbb{R}^n} \left(\frac{|f|}{t}\right)^r(wv)^r\right)^{\gamma/n}\\
	&=\left[\left(\frac{|f|}{t}\right)(wv)^{1-q/r}\left(\int_{\mathbb{R}^n} \left(\frac{|f|}{t}\right)^r(wv)^r\right)^{\gamma q/(nr)}\right]^{r/q}.
	\end{align*}
	Observe that for $0\leq t\leq 1$, $\xi(t^{r/q})=t^r$, and for $t>1$,
	\begin{align*}
	\xi(t^{r/q})&=t^r(1+\log t^{r/q})^\nu\\
	&=t^r\left(1+\frac{r}{q}\log t\right)^\nu,
	\end{align*}
	which implies $\xi(t^{r/q})\leq \Phi_\gamma(t)=t^r(1+\log^+t)^\nu$. Then we can estimate as follows
	\begin{align*}
	\xi\left(\frac{f_0^{r/q}\left(\int_{\mathbb{R}^n}f_0^r\right)^{\gamma/n}}{wv t}\right)&\leq \Phi_\gamma\left(\left(\frac{|f|}{t}\right)(wv)^{1-q/r}\left(\int_{\mathbb{R}^n} \left(\frac{|f|}{t}\right)^r(wv)^r\right)^{\gamma q/(nr)}\right)\\
	&\leq \Phi_\gamma\left(\left[\int_{\mathbb{R}^n}\Phi_\gamma\left(\frac{|f|}{t}\right)(wv)^r\right]^{\gamma q/(nr)}\right)\Phi_\gamma\left(\frac{|f|}{t}(wv)^{1-q/r}\right)
	\end{align*}
	Returning to \eqref{eq: eq1 - teo mixta para M_{gamma,Phi}, caso p=r} and setting $w=u^{1/q}$, the right hand side is bounded by
	\[ \Phi_\gamma\left(\left[\int_{\mathbb{R}^n}\Phi_\gamma\left(\frac{|f|}{t}\right)(wv)^r\right]^{\gamma q/(nr)}\right)\int_{\mathbb{R}^n} \Phi_\gamma\left(\frac{|f|}{t}(wv)^{1-q/r}\right)(wv)^q.\]
	Notice that $\Phi_\gamma(t^{1-q/r})t^q\leq \Psi(t)$. Therefore, the expression above is bounded by 
	\[\Phi_\gamma\left(\left[\int_{\mathbb{R}^n}\Phi_\gamma\left(\frac{|f|}{t}\right)\Psi(u^{1/q}v)\right]^{\gamma q/(nr)}\right)\int_{\mathbb{R}^n} \Phi_\gamma\left(\frac{|f|}{t}\right)\Psi(u^{1/q}v).\]
	To finish, observe that 
	\[t\Phi_\gamma(t^{\gamma q/(nr)})\lesssim t^{1+\gamma q/n}(1+\log^+ t)^\nu= t^{q/r}(1+\log^+ t)^{\delta q/r}=\varphi(t).\qedhere\]
\end{proof}



	\def\cprime{$'$}
	\providecommand{\bysame}{\leavevmode\hbox to3em{\hrulefill}\thinspace}
	\providecommand{\MR}{\relax\ifhmode\unskip\space\fi MR }
	\providecommand{\MRhref}[2]{%
		\href{http://www.ams.org/mathscinet-getitem?mr=#1}{#2}
	}
	\providecommand{\href}[2]{#2}

\end{document}